\input ./style/arxiv-vmsta.cfg
\documentclass[numbers,compress,v1.0.1]{vmsta}
\usepackage{vtexbibtags}

\volume{6}
\issue{4}
\pubyear{2019}
\firstpage{397}
\lastpage{417}
\aid{VMSTA141}
\doi{10.15559/19-VMSTA141}
\articletype{research-article}

\startlocaldefs
\newtheorem{lemma}{Lemma}
\newtheorem{prop}{Proposition}
\theoremstyle{definition}
\newtheorem{remark}{Remark}
\newtheorem{definition}{Definition}
\newcolumntype{d}[1]{D{.}{.}{#1}}

\allowdisplaybreaks
\ifdefined\HCode 
\def\index#1{}
\else 
\fi
\endlocaldefs

\begin{document}

\begin{frontmatter}
\pretitle{Research Article}

\title{Estimation of the drift parameter for the fractional stochastic heat equation via power variation}

\author{\inits{Z.}\fnms{Zeina}~\snm{Mahdi Khalil}\ead[label=e1]{zeina\_kh@outlook.fr}}
\author{\inits{C.}\fnms{Ciprian}~\snm{Tudor}\thanksref{cor1}\ead[label=e2]{ciprian.tudor@univ-lille.fr}}
\thankstext[type=corresp,id=cor1]{Corresponding author.}
\address{Laboratoire Paul Painlev\'{e},
\institution{Universit\'{e} de Lille}, CNRS, UMR 8524,
F-59655~Villeneuve d'Ascq, \cny{France}}


\markboth{Z. Mahdi Khalil, C. Tudor}{Estimation of the drift parameter for the fractional stochastic heat equation via power variation}

\begin{abstract}
We define power variation estimators for the drift parameter of the
stochastic heat equation with the fractional Laplacian and 
an additive
Gaussian noise which is white in time and white or correlated in space.
We prove that these estimators are consistent and asymptotically normal
and we derive their rate of convergence under the Wasserstein metric.
\end{abstract}

\received{\sday{8} \smonth{4} \syear{2019}}
\revised{\sday{22} \smonth{7} \syear{2019}}
\accepted{\sday{11} \smonth{9} \syear{2019}}
\publishedonline{\sday{3} \smonth{10} \syear{2019}}

\end{frontmatter}

\section{Introduction}%
\label{sec1}
The purpose of this work is to estimate the drift parameter
$\theta >0$ of the fractional stochastic heat equation\index{fractional ! stochastic heat equation}\index{stochastic heat equation fractional}
%
\begin{equation}
\label{intro-1}
\frac{\partial u _{\theta } }{\partial t} (t,x)= -\theta (-\Delta )
^{\frac{\alpha }{2}}u_{\theta } (t,x)+ \dot{W} (t,x),\quad
 t\geq 0, x\in \mathbb{R},
\end{equation}
with vanishing initial conditions, where $(-\Delta ) ^{
\frac{\alpha }{2}}$ denotes the fractional Laplacian\index{fractional ! Laplacian} of order
$\alpha \in (1, 2]$, $\theta >0$ and $W$ is a Gaussian noise\index{Gaussian noise} which is
white in time and white or correlated in space.

The parameter estimation for stochastic partial differential equations
(SPDEs in the sequel) constitutes a research direction of wide interest
in probability theory and mathematical statistics. We refer, among many
others, to the recents surveys \cite{Lot} and \cite{Cia}.
On the other side, there are relatively few works that consider the
solution to a SPDE observed at discrete points in time and/or in space.
Among the first works in this direction, we refer to \cite{Mo} and
\cite{Ma} for the maximum likelihood and least square estimators
for parabolic, respectively elliptic type SPDEs driven by a space-time
white noise. The study in \cite{Ma} has been then extended in
\cite{BT1}, by adding a time-varying volatility in the noise term and
by using power variation techniques to estimate the parameter of the
model. Other recent works on parameter estimates for discretely sampled
SPDEs via power variations\index{power variations} are \cite{CH,Chong,BT2,PoTr} and \cite{ZZ}.

In this paper, we extend the above results into two directions. Firstly,
we replace the standard Laplacian operator used in all the above
references by a fractional Laplacian.\index{fractional ! Laplacian} On the other hand, we consider a
simpler form, comparing to \cite{BT1,Ma}, of the
differential operator. Secondly, we also consider a noise term which is
correlated in space. Our purpose is to propose power variation type
estimators for the drift parameter in the stochastic model
(\ref{intro-1}), based on discrete observations of the solution in time
or in space, and to analyze the consistency and the limit distribution
of the estimators by taking advantage of the link between the solution
and the fractional Brownian motion.\index{fractional ! Brownian motion} Our approach to construct and
analyze the estimators for the drift parameter is based on the
asymptotic behavior of the $q$-variations of the mild solution\index{mild solution} to
(\ref{intro-1}). It is well known (see, e.g., \cite{FKM,MahTu,T}) that there exists a strong link between
the law of this mild solution\index{mild solution} with $\theta =1$ and the fractional
Brownian motion\index{fractional ! Brownian motion} and related processes. We will use this connection in
order to deduce the behavior of the $q$-variations (of suitable order
$q$) of the solutions to (\ref{intro-1}) and to prove the consistency,
asymptotic normality and Berry--Ess\'{e}en bounds under the Wasserstein
distance\index{Wasserstein distance} for the associated estimators. For the situation when $W$ is
a space-time white noise, we will obtain two estimators for the drift
parameter: one based on the temporal variations and one based of the
spatial variations of the mild solution\index{mild solution} $u_{\theta }$. Similarly, two
estimators are defined when the Gaussian noise\index{Gaussian noise} $W$ is white in time and
colored in space (with the spatial covariance given by the Riesz
kernel\index{Riesz kernel}). Even if the order of the variations which appear in the
definition of the estimator is different in the four cases (this order
may depend on the parameter $\alpha $ of the fractional Laplacian\index{fractional ! Laplacian} and/or
on the spatial correlation), all the estimators are asymptotically
normal, they have the same rate of convergence of order $n ^{-
\frac{1}{2}}$ and they have the same distance to the Gaussian
distribution. The case of the standard Laplacian\index{standard Laplacian} (i.e., $\alpha =2$) has
been studied in \cite{PoTr}.

We organize the paper as follows. In Section~\ref{sec2} we present general facts
on the stochastic heat equation\index{stochastic heat equation} with the fractional Laplacian\index{fractional ! Laplacian} and the
behavior of the variations of the perturbed fractional Brownian motion.\index{fractional ! Brownian motion}
In Section~\ref{sec3} we discuss the drift parameter estimation for the
fractional heat equation\index{heat equation} with a space-time white noise while in Section~\ref{sec4} we treat the case when the noise is correlated in space.

We will denote by $c$, $C$ a generic positive constant that may change
from line to line (or even inside of the the same line). By
$\to ^{(d)} $ we denote the convergence in distribution while
$\equiv ^{(d)} $ stands for the equivalence of finite dimensional
distributions.\index{finite dimensional distributions}

\section{The fractional heat equation\index{heat equation} driven by a space-time white noise}%
\label{sec2}
We start by treating the fractional stochastic heat equation\index{fractional ! stochastic heat equation}\index{stochastic heat equation fractional} with
a space-time white noise. We recall the basic properties of the solution,
its relation with the fractional Brownian motion\index{fractional ! Brownian motion} and then we discuss the
estimation of the drift parameter $\theta $ via the $q$-variations.

\subsection{General properties of the solution}%
\label{sec2.1}
On the standard probability space $\left ( \Omega , \mathcal{F},
P\right )$, we consider a centered Gaussian field
$\left ( W(t, A), t\geq 0, A \in \mathcal{B} _{b} (\mathbb{R}) \right )$
with covariance
%
\begin{equation}
\label{covwn}
\mathbf{E} W(t, A) W(s, B)= (s\wedge t) \lambda (A \cap B) \quad \mbox{ for
every } s,t\geq 0, A,B\in \mathcal{B} _{b} (\mathbb{R}),
\end{equation}
where $\lambda $ denotes the Lebesgue measure on $\mathbb{R}$ and
$\mathcal{B} _{b} (\mathbb{R}) $ is the class of bounded Borel subsets
of $\mathbb{R}$. The Gaussian field $W$ is usually called the space-time
white noise.

We will consider the stochastic heat equation\index{stochastic heat equation}
%
\begin{equation}
\label{1}
\frac{\partial u _{\theta } }{\partial t} (t,x)= -\theta (-\Delta )
^{\frac{\alpha }{2}}u_{\theta } (t,x)+ \dot{W} (t,x),\quad
 t\geq 0, x\in \mathbb{R},
\end{equation}
with vanishing initial condition $u(0, x)= 0$ for every $ x\in
\mathbb{R}$. In the above equation, $ (-\Delta ) ^{\frac{\alpha }{2}}$
represents the fractional Laplacian\index{fractional ! Laplacian} of order $\alpha $. We will assume
in the sequel that $\alpha \in (1, 2]$. We refer to \cite{DD,Ja1,Ja2,Jiang} for the precise
definition and other properties of the fractional Laplacian operator.\index{fractional ! Laplacian}
We will denote its Green kernel (or the fundamental solution) by
$ G_{\alpha }$, which represents the deterministic kernel that solves
the heat equation\index{heat equation} without noise $\frac{\partial }{\partial t} u(t,x) =
-(-\Delta ) ^{\frac{\alpha }{2}}u(t,x) $. It is know from the above
references that for $t>0$, $x\in \mathbb{R}$
%
\begin{equation}
\label{G}
G_{\alpha }(t,x) =\int _{\mathbb{R}} e ^{it \xi -t\vert \xi \vert ^{
\alpha }} d\xi .
\end{equation}

It is an immediate
conclusion that the fundamental solution associated to the
operator $-\theta (-\Delta ) ^{\frac{\alpha }{2}}u_{\theta } (t,x)$ is
$G_{\alpha } (\theta t, x)$.

The solution to (\ref{1}) is understood in the mild sense, i.e.,
%
\begin{equation}
\label{sol1}
u_{\theta } (t, x)= \int _{0} ^{t} \int _{\mathbb{R}} G_{\alpha } (
\theta (t-s), x-y) W (ds, dy),
\end{equation}
where the stochastic integral $W (ds, dy)$ is the usual Wiener integral\index{Wiener integral}
with respect to the space-time white noise, which satisfies the isometry
\begin{equation*}
\mathbf{E} \left (\int _{0} ^{T} \int _{\mathbb{R}} H(s,y) W(ds,
dy)\right ) ^{2}= \int _{0} ^{T} \int _{\mathbb{R}} H(s,y)^{2} dyds
\end{equation*}
for every $T>0$ and for every measurable square integrable function
$H$.

For $\theta =1$, the solution to the heat equation\index{heat equation} (\ref{1}) has been
studied in \cite{MahTu}. This solution exists only if the spatial
dimension is $d=1$, and it is connected to the bifractional Brownian
motion.\index{bifractional Brownian motion} Recall that (see \cite{HV,RuTu}), given
constants $H\in (0,1)$ and $K\in (0,1]$, the bifractional Brownian
motion\index{bifractional Brownian motion} (bi-fBm for short) $(B^{H,K}_{t})_{t \geq 0}$ is a centered
Gaussian process with covariance
%
\begin{equation}
\label{cov-bi}
R^{H,K}(t,s) := R(t,s)= \frac{1}{2^{K}}\left ( \left (
t^{2H}+s^{2H}\right )
^{K} -\vert t-s \vert ^{2HK}\right ),
\hskip0.5cm s,t \geq 0.
\end{equation}
In particular, for $K=1$, $B ^{H,}:= B ^{H,1}$ is the fractional
Brownian motion\index{fractional ! Brownian motion} (fBm in the sequel) with the Hurst parameter $H\in (0,1)$.

Let us recall some of the results in \cite{MahTu} which will be
needed in the sequel.
\begin{itemize}%
\item
The mild solution\index{mild solution} (\ref{sol1}) is well-defined. For every $x\in
\mathbb{R}$, the process  ($ u_{1}(t,x)$, $t\geq 0$) coincides
in distribution, modulo a constant, with the bifractional Brownian
motion,\index{bifractional Brownian motion} i.e.,
\begin{equation*}
\left ( u_{1}(t, x), t\geq 0\right ) \equiv ^{(d)} \left (c_{2, \alpha
} B_{t} ^{\frac{1}{2}, 1-\frac{1}{\alpha }}, t\geq 0\right ),
\end{equation*}
where $B ^{\frac{1}{2}, 1-\frac{1}{\alpha }}$ is a bifractional Brownian
motion\index{bifractional Brownian motion} with the Hurst parameters $H=\frac{1}{2}$\index{Hurst parameters} and $K=1-\frac{1}{\alpha
}$ and
%
\begin{equation}
\label{c2a}
c_{2, \alpha } ^{2}= c_{1, \alpha } 2 ^{1-\frac{1}{\alpha }} \mbox{ with
} c_{1, \alpha }=\frac{1}{2\pi (\alpha -1) }\Gamma \left ( \frac{1}{
\alpha }\right ).
\end{equation}
\item
For every $t\geq 0$, we have (see Proposition 3.1 in \cite{FKM})
%
\begin{equation}
\label{2m-1}
\left ( u_{1}(t,x), x\in \mathbb{R}\right ) \equiv ^{(d)} \left ( m
_{\alpha } B ^{ \frac{\alpha -1}{2}} (x) + S_{t}(x), x\in \mathbb{R}
\right ),
\end{equation}
where $B^{\frac{\alpha -1}{2}}$ is a fractional Brownian motion\index{fractional ! Brownian motion} with
the Hurst parameter $\frac{\alpha -1}{2}\in [0, \frac{1}{2}] $,
$(S_{t}(x))_{x\in \mathbb{R}} $ is a centered Gaussian process with
$C^{\infty } $ sample paths and $m_{\alpha } $ is an explicit numerical
constant.

\end{itemize}

The above facts, combined with the decomposition (\ref{5m-1}) of the
bifractional Brownian motion,\index{bifractional Brownian motion} show that the solution to the heat
equation\index{heat equation} can be expressed as the sum of a fBm and a smooth process (we
will call this sum as a perturbed fractional Brownian motion\index{fractional ! Brownian motion}).

\subsection{Variations of the perturbed fractional Brownian motion\index{fractional ! Brownian motion}}%
\label{sec2.2}
Since the process (\ref{sol1}) is connected to the perturbed fBm\index{perturbed fBm} (i.e.,
the sum of a fBm and a smooth Gaussian process\index{Gaussian process}), let us recall some
facts concerning the asymptotic behavior of the variation of the
perturbed fBm.\index{perturbed fBm} Some of the below results are directly taken from
\cite{MahTu} while those concerning the rate of convergence under the
Wasserstein distance\index{Wasserstein distance} are deduced from \cite{NP-book}.

We first define the notion of \emph{(exact) $q$-variation} for
stochastic processes.

\begin{definition}
\label{def1}
Let $A_{1}<A_{2}$, and for $n\geq 1$, let $t_{i}=A_{1} + \frac{i}{n} (A
_{2}-A_{1}) $ for $i=0,\ldots,n$. A continuous stochastic process
$(X_{t}) _{t\geq 0} $ admits a $q$-variation (or a variation of order
$q$) over the interval $[A_{1}, A_{2} ]$ if the sequence
\begin{equation*}
S^{n, q} _{[A_{1}, A_{2}]} (X): = \sum _{i=0} ^{n-1} \left | X_{t_{i+1}}-
X_{t_{i}} \right | ^{q}
\end{equation*}
converges in probability as $n\to \infty $. The limit, when it exists,
is called the exact $q$-variation of $X$ over the interval $[A_{1}, A
_{2}]$.

\end{definition}

If $[A_{1}, A_{2}]=[0,t]$, we will simply denote $S^{n, q} _{t}(X):=S
^{n, q} _{[0,t]} (X)$. Moreover, if $t=1$, we denote $ S ^{q,n}(X):=S
^{n, q} _{t}(X)$. In the case $q=2$ the limit of $S ^{2,n}$ is called
the quadratic variation, while for $q=3$ we have the cubic variation.

Let us recall the following result (see \cite{MahTu}) concerning
the exact variation of the perturbed fractional Brownian motion,\index{fractional ! Brownian motion} i.e.,
the sum of a fBm and a smooth Gaussian process.\index{Gaussian process} In the rest of this
section, we will fix an interval $[A_{1}, A_{2}]$ with $A_{1}<A_{2}$ and
a partition $t_{j}= A_{1}+ \frac{j}{n} (A_{2}-A_{1}) $, $n\geq 1$, $j=0,\ldots, n$, of this interval. Also, we denote by $Z$ a standard normal
random variable, and $\mu _{q} =\mathbf{E} Z ^{q}$ for $q\geq 1$. Define
$\sigma _{H, q } ^{2}= q! \sum _{v\in \mathbb{Z}} \rho _{H}(v) ^{q}$, with
$\rho _{H}(v) =\frac{1}{2} \left ( \vert v+1\vert ^{2H}+ \vert v-1\vert
^{2H} -2\vert v\vert ^{2H}\right )$ for $v\in \mathbb{Z}$.

\vskip0.2cm

\begin{lemma}
\label{l1}
Let $ (B ^{H}_{t}) _{t\geq 0}$ be a fBm with $H\in (0, \frac{1}{2}]$ and
consider a centered Gaussian process $(X_{t}) _{t\geq 0}$ such that
%
\begin{equation}
\label{23a-1}
\mathbf{E} \left | X_{t}- X_{s} \right | ^{2} \leq C \vert t-s\vert
^{2}\quad  \mbox{ for every } s,t \geq 0.
\end{equation}
Define
\begin{equation*}
Y^{H} _{t}= aB^{H}_{t} + X_{t}\quad  \mbox{ for every } t\geq 0
\end{equation*}
with $a\neq0$.
\begin{enumerate}%
\item
The process $Y$ has $\frac{1}{H}$-variation over the interval
$[A_{1}, A_{2}]$ which is equal to
\begin{equation*}
a^{-\frac{1}{H}}\mathbf{E} \vert Z \vert ^{1/H} (A_{2}-A_{1}).
\end{equation*}
\item
Let
%
\begin{equation}
\label{5a-1}
V_{q, n}(Y ^{H} ):=\sum _{i=0} ^{n-1} \left [ \frac{n ^{ Hq}}{ (A_{2}-A
_{1}) ^{qH}a^{q}} ( Y^{H} _{t_{i+1}} - Y ^{H} _{ t_{i}}) ^{q} - \mu
_{q} \right ].
\end{equation}
Then, if $H\in (0, \frac{1}{2})$ and $q\geq 2$ is an integer,
%
\begin{align}
\label{c11}
&\frac{1}{\sqrt{n}}V_{q,n}(Y ^{H} )=\frac{1}{\sqrt{n}}\sum _{i=0}
^{n-1} \left [ \frac{n ^{ Hq}}{ (A_{2}-A_{1}) ^{qH}a^{q}} ( Y^{H}
_{t_{i+1}} - Y ^{H} _{ t_{i}}) ^{q} - \mu _{q} \right ]
\nonumber \\
&\to ^{(d)} N (0,
\sigma _{ H, q } ^{2}).
\end{align}
If $H=\frac{1}{2}$, $q=2$ and the process $(X_{t}) _{t\geq 0}$ is
adapted to the filtration generated by $B$, then
%
\begin{equation}
\label{c12}
\frac{1}{\sqrt{n} } V_{2,n}( Y ^{H})=\frac{1}{\sqrt{n} } \sum _{i=0}
^{n-1} \left [ \frac{n }{ (A_{2}-A_{1}) a^{2}} ( Y ^{\frac{1}{2}}
_{t_{i+1}} - Y ^{\frac{1}{2}} _{ t_{i}}) ^{2} - 1\right ] \to ^{(d)} N
(0, \sigma _{ \frac{1}{2}, 2 } ^{2}).
\end{equation}

\end{enumerate}
\end{lemma}

Using the recent Stein--Malliavin theory, it is also possible to deduce
the rate of convergence in the above Central Limit Theorem (CLT in the
sequel) under the Wasserstein distance.\index{Wasserstein distance} Before stating and proving the
result, let us briefly recall the definition of the Wasserstein
distance.\index{Wasserstein distance} The Wasserstein distance\index{Wasserstein distance} between the laws of two $
\mathbb{R} ^{d}$-valued random variables $F$ and $G$ is defined as
%
\begin{equation}
\label{dw}
d_{W} (F, G)= \sup _{h\in \mathcal{A}}\left | \mathbf{E}h(F)-
\mathbf{E}h(G)\right |
\end{equation}
where $\mathcal{A}$ is the class of Lipschitz continuous function
$h:\mathbb{R} ^{d} \to \mathbb{R}$ such that $\Vert h\Vert _{Lip}
\leq 1$, where
\begin{equation*}
\Vert h\Vert _{Lip}= \sup _{x, y\in \mathbb{R} ^{d}, x\neq y} \frac{
\vert h(x)-h(y)\vert }{\Vert x-y\Vert _{\mathbb{R} ^{d}}}.
\end{equation*}

\begin{prop}
\label{pp1}
Assume $H\leq \frac{1}{2}$. Let $Y^{H}$ be as in Lemma \ref{l1} and let
$V_{q, n}(Y ^{H} )$ be given by (\ref{5a-1}). Then for $n$ large and
with $\sigma _{ H, q }$ from (\ref{c11}),
\begin{equation*}
d_{W} \left ( \frac{1}{\sqrt{n}}V_{q,n}(Y ^{H} ), N (0, \sigma _{ H,
q } ^{2})\right ) \leq C \frac{1}{\sqrt{n}}.
\end{equation*}
\end{prop}
\begin{proof}
From the proof of Lemma 2.1 in \cite{MahTu}, we can express the
variation of $Y^{H}$ and the variation of the fBm $B^{H}$ plus a rest
term, i.e.,
\begin{equation*}
\frac{1}{\sqrt{n}} V_{q,n}(Y ^{H} )=\frac{1}{\sqrt{n}} V_{q,n}(B
^{H} )+ R_{n},
\end{equation*}
where $R_{n}$ satisfies, for every $n\geq 1$,
%
\begin{equation}
\label{5a-2}
\mathbf{E} \vert R_{n} \vert \leq cn ^{H-1}.
\end{equation}
By the definition of the Wasserstein distance,\index{Wasserstein distance} we can write
\begin{eqnarray*}
&&d_{W} \left ( \frac{1}{\sqrt{n}}V_{q,n}(Y ^{H} ), N (0,
\sigma _{ H, q } ^{2})\right )
\\
&\leq & d_{W} \left ( \frac{1}{\sqrt{n}}V_{q,n}(B ^{H} ), N (0,
\sigma _{ H, q } ^{2})\right ) + d_{W} \left ( \frac{1}{\sqrt{n}}V
_{q,n}(Y ^{H} ), \frac{1}{\sqrt{n}}V_{q,n}(B ^{H} )\right )
\\
&\leq & d_{W} \left ( \frac{1}{\sqrt{n}}V_{q,n}(B ^{H} ), N (0,
\sigma _{ H, q } ^{2})\right )+ \mathbf{E} \vert R_{n} \vert .
\end{eqnarray*}
In order to estimate $d_{W} ( \frac{1}{\sqrt{n}}V_{q,n}(B ^{H} ), N
(0, \sigma _{ H, q } ^{2}))$, we will use the chaos expansion of the
random variable $V_{q,n}(B ^{H} )$ and several results in
\cite{NP-book}. Notice that (see, e.g., the proof of Corollary 3 in
\cite{NNT}),
\begin{equation*}
V_{q,n}(B ^{H} ) = \sum _{k=1} ^{q} k! C_{q} ^{k} \mu _{q-k} \sum _{i=0}
^{n-1} H_{k}\left ( \frac{ n ^{HK}}{(A_{2}- A_{1}) ^{HK}} \left ( B
^{H}_{t_{i+1} }- B^{H}_{t_{i}}\right ) \right ),
\end{equation*}
where $H_{k}$ is the $k$-th probabilists' Hermite polynomial
\[
H_{k}(x)= (-1) ^{k} e ^{-\frac{x^{2}}{2}} \frac{d^{n}}{dx^{n}}\left (
e ^{-\frac{x^{2}}{2}} \right )
\]
for $k\geq 1$ with $H_{0}(x)=1$. We
know from \cite{NP-book} that the vector
\begin{equation*}
( F_{1,n}, F_{2, n},\ldots, F_{q,n}):= \left ( \frac{1}{\sqrt{n}}
\sum _{i=0} ^{n-1} H_{k}\left ( \frac{ n ^{HK}}{(A_{2}- A_{1}) ^{HK}}
\left ( B^{H}_{t_{i+1} }- B^{H}_{t_{i}}\right ) \right ) \right )
_{k=1,\ldots , q}
\end{equation*}
converges in distribution to a centered Gaussian vector with diagonal
covariance matrix $C$ (the explicit expression of $C$ can be found in
\cite{NP-book}, it is not needed in our work). Moreover,
Proposition 6.2.2 and Corollary 7.4.3 in \cite{NP-book} imply
that
\begin{equation*}
d_{W} \left ( (F_{k, n})_{k=1,\ldots, q}, N(0, C) \right ) \leq c\frac{1}{
\sqrt{n}}.
\end{equation*}
This will easily lead to
%
\begin{equation}
\label{5a-3}%
d_{W} \left ( \frac{1}{\sqrt{n}}V_{q,n}(B ^{H} ), N (0, \sigma _{ H,
q } ^{2})\right )\leq c\frac{1}{\sqrt{n}}.
\end{equation}
Since $H\leq \frac{1}{2}$, we obtain the conclusion via (\ref{5a-2}) and
(\ref{5a-3}).
\end{proof}

\subsection{Estimators for the drift parameter}%
\label{sec2.3}
Our purpose is to estimate the parameter $\theta >0$ based on the observations
of the process~$u_{\theta } $. We will define two estimators: the first
is based on the temporal variations of the process $u_{\theta }$ while
the second is constructed via its variation in space. Their behavior is
strongly related to the law of the process $u_{\theta }$, therefore we
start by analyzing the distribution of this Gaussian process.\index{Gaussian process}

\subsubsection{The law of the solution}
\label{sec2.3.1}
Let $G_{\alpha }(t,x)$ be the Green kernel associated to the operator
$- (-\Delta ) ^{\frac{\alpha }{2}}$. Then the Green kernel associated
to the operator operator $-\theta (-\Delta ) ^{\frac{\alpha }{2}}$ is
\begin{equation*}
G_{\alpha }(\theta t,x).
\end{equation*}

\begin{lemma}
\label{ll1}
Suppose that the process $\left (u_{\theta } (t,x), t\geq 0, x\in
\mathbb{R}\right )$ satisfies (\ref{1}). Define
%
\begin{equation}
\label{uv}
v_{\theta }(t,x):= u_{\theta } \left (\frac{t}{\theta }, x\right ),
\hskip0.4cm t\geq 0, x\in \mathbb{R}.
\end{equation}
Then the process $\left (v_{\theta } (t,x), t\geq 0,
x\in \mathbb{R}\right )$ satisfies
%
\begin{equation}
\label{2}
\frac{\partial v_{\theta }}{\partial t} (t,x)= - (-\Delta ) ^{\frac{
\alpha }{2}}v_{\theta }(t,x)+(\theta ) ^{-\frac{1}{2}}
\dot{\widetilde{W}} (t,x),
\hskip0.4cm t\geq 0, x\in \mathbb{R},
\end{equation}
with $ v_{\theta } (0, x)=0$ for every $x\in \mathbb{R}$, where
$\dot{\widetilde{W}}$ is a space-time white noise, i.e., a centered
Gaussian random field with covariance (\ref{covwn}).
\end{lemma}
\begin{proof}
From (\ref{sol1}), we have for every $t\geq 0$, $x\in \mathbb{R}$,
\begin{eqnarray*}
v_{\theta } (t, x) =u_{\theta } \left ( \frac{t}{\theta }, x\right )
&=&
\int _{0} ^{\frac{t}{\theta } } \int _{\mathbb{R}} G_{\alpha } (t-
\theta s, x-y) W(ds, dy)
\\
&=& \int _{0} ^{t} \int _{\mathbb{R}} G_{\alpha } (t-s, x-y) W (d\frac{s}{
\theta }, dy)
\\
&=&\theta ^{-\frac{1}{2} } \int _{0} ^{t} \int _{\mathbb{R}} G_{\alpha }
(t-s, x-y) \tilde{W} (ds, dy),
\end{eqnarray*}
where, for $t\geq 0$, $A\in \mathcal{B}(\mathbb{R} )$, we denoted
$\tilde{W} (t, A):= \theta ^{\frac{1}{2}} W \left (\frac{t}{\theta },
A\right )$. Notice that $\tilde{W}$ has the same finite dimensional
distributions\index{finite dimensional distributions} as $W$, due to the scaling property of the white noise.
\end{proof}

We can deduce the law of the process $u_{\theta }$ in time and space.

\begin{prop}
\label{pp2}
For every $x \in \mathbb{R}$ and $\theta >0$, we have
\begin{equation*}
\left ( u_{\theta } (t,x), t\geq 0\right ) \equiv ^{(d)} \left (
\theta ^{-\frac{1}{2\alpha }} c_{2, \alpha } B_{t} ^{\frac{1}{2}, 1-\frac{1}{
\alpha }}, t\geq 0\right ),
\end{equation*}
where $B ^{\frac{1}{2}, 1-\frac{1}{\alpha }} $ is a bifractional
Brownian motion\index{bifractional Brownian motion} with parameters $H=\frac{1}{2} $ and $K= 1-\frac{1}{
\alpha }$ and $c_{2, \alpha }$ is given by (\ref{c2a}).
\end{prop}
\begin{proof}
Fix $x\in \mathbb{R}$ and $\theta >0$. Then for every $s, t\geq 0$, we
have
\begin{align*}
& \mathbf{E} u_{\theta } (t, x) u_{\theta } (s, x) = \mathbf{E} v
_{\theta } (\theta t, x) v_{\theta }(\theta s, x)
\\
={}& \theta ^{-1} \mathbf{E} u_{1} (\theta t, x) u_{1} (\theta s, x) =
\theta ^{-1} c_{1, \alpha } \left [ (\theta t+\theta s) ^{1-\frac{1}{
\alpha }} -\vert \theta t -\theta s \vert ^{1-\frac{1}{\alpha }} \right ]
\\
={}& \theta ^{-\frac{1}{\alpha }} c_{2, \alpha } ^{2} \mathbf{E} B ^{
\frac{1}{2}, 1-\frac{1}{\alpha }} _{t} B ^{\frac{1}{2}, 1-\frac{1}{
\alpha }} _{s}.\qedhere
\end{align*}

\end{proof}

\begin{prop}
\label{pp3}
For every $t\geq 0$, $\theta >0$, we have the following equality in
distribution
\begin{equation*}
\left ( u_{\theta }(t,x), x\in \mathbb{R}\right ) \equiv ^{(d)} \left (
\theta ^{-\frac{1}{2}} m_{\alpha } B ^{ \frac{\alpha -1}{2}} (x) + S
_{\theta t}(x), x\in \mathbb{R} \right ) ,
\end{equation*}
where $B^{\frac{\alpha -1}{2}}$ is a fractional Brownian motion\index{fractional ! Brownian motion} with
the Hurst parameter $\frac{\alpha -1}{2}\in (0, \frac{1}{2}] $,
$(S_{\theta t}(x))_{x\in \mathbb{R}} $ is a centered Gaussian process
with $C^{\infty } $ sample paths and $m_{\alpha }$ from (\ref{2m-1}).

\end{prop}
\begin{proof}
The result is immediate since for every $t>0$, $\theta >0$
\begin{equation*}
\left ( u_{\theta }(t,x), x\in \mathbb{R}\right )= \left ( v_{\theta
}(\theta t,x), x\in \mathbb{R}\right )\equiv ^{(d)} \theta ^{-
\frac{1}{2} }\left ( u_{1} (\theta t, x), x\in \mathbb{R}\right )
\end{equation*}
\begin{equation*}
\equiv ^{(d)} \left ( \theta ^{-\frac{1}{2}} m_{\alpha } B ^{ \frac{
\alpha -1}{2}} (x) + S_{\theta t}(x), x\in \mathbb{R} \right ),
\end{equation*}
where we used (\ref{2m-1}).
\end{proof}

Notice that the Hurst parameter of the fBm in Proposition \ref{pp3} may
be $\frac{1}{2}$ if $\alpha =2$.

\subsubsection{Estimators based on the temporal variation}%
\label{sec2.3.2}
Proposition \ref{pp2} indicates that the process $u_{\theta }$ behaves
as a bi-fBm in time. Recall the following connection between the fBm and
the bi-fBm (see \cite{LN}): Let $H\in (0, 1)$, $K\in (0, 1]$. If
$ (B ^{HK} _{t}) _{t\geq 0} $ is a fBm with the Hurst parameter $HK$ and
$ ( B^{H, K}_{t}) _{t\geq 0} $ is a bi-fBm, then
%
\begin{equation}
\label{5m-1}
\left ( C _{1} X ^{H, K} _{t}+ B_{t} ^{H, K}, t\geq 0\right ) \equiv
^{(d)} \left ( C_{2} B _{t} ^{HK}, t\geq 0\right ),
\end{equation}
with $C_{1}>0$ and $ C_{2}= 2^{\frac{ 1-K}{2}}$. In (\ref{5m-1}),
$X ^{H, K}$ is a Gaussian process,\index{Gaussian process} independent of $B^{H, K}$ with
$C^{\infty }$ sample paths. In particular, it satisfies (\ref{23a-1}).
Therefore, the bi-fBm is a perturbed fBm\index{perturbed fBm} and the same holds true for the
solution $(u_{\theta }(t, x), t\geq 0)$, by Proposition \ref{pp2}.
Therefore, we obtain, by using the notation $t_{j}= A_{1}+
\frac{j}{n} (A_{2}-A_{1}) $, $n\geq 1$, $j=0,\ldots , n$, the following lemma.

\begin{lemma}
Let $u_{\theta } $ be the solution to (\ref{1}). Then for every
$x\in \mathbb{R}$,
%
\begin{eqnarray}
\label{28f-3}
&&S^{n, \frac{2\alpha }{\alpha -1}} _{[A_{1}, A_{2}]} : = \sum _{i=0}
^{n-1} \left | u_{\theta }(t_{j+1}, x)- u_{\theta } (t_{j}, x) \right |
^{\frac{2\alpha }{\alpha -1}}
\nonumber \\
&& \to _{n\to \infty }c_{2, \alpha } ^{\frac{2\alpha }{\alpha -1}} 2
^{\frac{1}{\alpha -1}}\mu _{\frac{2\alpha }{\alpha -1}} (A_{2}-A_{1})
\vert (\theta ) \vert ^{\frac{-1}{\alpha -1}}
\end{eqnarray}
in probability.
\end{lemma}

Relation (\ref{28f-3}) motivates the definition of the following
estimator for the parameter $\theta >0$ of the model (\ref{1}):
%
\begin{eqnarray}
\label{hatt}
&&\widehat{\theta }_{n,1}
\nonumber
\\
&=& \left ( \left ( c_{2, \alpha } ^{\frac{2\alpha }{\alpha -1}} 2^{\frac{1}{
\alpha -1}}\mu _{\frac{2\alpha }{\alpha -1}} (A_{2}-A_{1})\right )
^{-1}\sum _{i=0} ^{n-1} \left | u_{\theta }(t_{j+1}, x)-u_{\theta }(t
_{j}, x)\right | ^{\frac{2\alpha }{\alpha -1} } \right ) ^{1-\alpha }
\nonumber \\
&=& \ \left ( c_{2, \alpha } ^{\frac{2\alpha }{\alpha -1}} 2^{\frac{1}{
\alpha -1}}\mu _{\frac{2\alpha }{\alpha -1}} (A_{2}-A_{1})\right )
^{\alpha -1}\left ( S^{n, \frac{2\alpha }{\alpha -1}}(u_{\theta } (
\cdot , x)) \right )^{1-\alpha },
\end{eqnarray}
and so
%
\begin{equation}
\label{5n-1}
\widehat{\theta } _{n,1} ^{\frac{1}{1-\alpha }} = \frac{1}{ c_{2,
\alpha } ^{\frac{2\alpha }{\alpha -1}} 2^{\frac{1}{\alpha -1}}
\mu _{\frac{2\alpha }{\alpha -1}} (A_{2}-A_{1})} \sum _{i=0} ^{n-1}
\left | u_{\theta }(t_{j+1}, x)-u_{\theta }(t_{j}, x)\right | ^{\frac{2
\alpha }{\alpha -1} }.
\end{equation}

We 
will prove the consistency and the asymptotic normality of the above
estimator.

\begin{prop}
\label{pp4}
Assume $q:=\frac{2\alpha }{\alpha -1}$ is an even integer and consider
the estimator $\widehat{\theta }_{n,1} $ defined by (\ref{hatt}). Then
$\widehat{\theta }_{n,1} \to _{n \to \infty }\theta $ in probability and
%
\begin{equation}
\label{5m-2}
\sqrt{n} \left [ \widehat{\theta }_{n,1} ^{\frac{1}{1-\alpha } } -
\theta ^{\frac{1}{1-\alpha } }\right ] \to ^{(d)}N(0, s_{1,\theta ,
\alpha } ^{2}) \mbox{ with }s_{1, \theta , \alpha } ^{2} =
\sigma _{\frac{1}{q}, q} ^{2}\theta ^{\frac{2}{1-\alpha }}
\mu _{\frac{2\alpha }{\alpha -1} } ^{-2}.
\end{equation}
Moreover, for $n$ large enough
\begin{equation*}
d_{W}\left ( \sqrt{n} \left [ \widehat{\theta }_{n,1} ^{\frac{1}{1-
\alpha } } -\theta ^{\frac{1}{1-\alpha } }\right ] , N(0, s_{\theta ,
\alpha } ^{2})\right ) \leq c\frac{1}{\sqrt{n}}.
\end{equation*}
\end{prop}
\begin{proof}
From Proposition \ref{pp2} and the relation between the fBm and the bi-fBm
(\ref{5m-1}), we obtain that
\begin{equation*}
\left ( u_{\theta } (t,x) + c_{2, \alpha } \theta ^{-\frac{1}{2\alpha
}} X _{t} \right ) \equiv ^{(d)} c_{2, \alpha }
\theta ^{-\frac{1}{2\alpha }} 2 ^{\frac{1}{2\alpha }} B ^{\frac{\alpha
-1}{2\alpha }},
\end{equation*}
where $ B ^{\frac{\alpha -1}{2\alpha }} $ is a fBm with the Hurst parameter
$\frac{\alpha -1}{2\alpha } \in (0, \frac{1}{2})$. Therefore,
$u_{\theta }$ is a perturbed fBm\index{perturbed fBm} and we obtain, by taking $H=\frac{
\alpha -1}{2\alpha }$ and $q=\frac{1}{H} = \frac{2\alpha }{\alpha -1},
$
\begin{align*}
&\frac{1}{\sqrt{n}} \sum _{i=0} ^{n-1} \left [\frac{ n
\theta ^{\frac{1}{\alpha -1}} }{ c_{2, \alpha } ^{\frac{2\alpha }{
\alpha -1}} 2^{\frac{1}{\alpha -1}} (A_{2}-A_{1}) }\left ( u_{\theta
}(t_{j+1}, x)-u_{\theta }(t_{j}, x)\right ) ^{\frac{2\alpha }{\alpha
-1} }- \theta ^{\frac{1}{1-\alpha } }\right ]
\\
&\to ^{(d)} N(0,
\sigma _{\frac{1}{q}, q} ^{2}).
\end{align*}
This means
\begin{equation*}
\sqrt{n} \mu _{\frac{2\alpha }{\alpha -1} } \theta ^{\frac{1}{\alpha
-1} } \left [ \widehat{\theta }_{n,1} ^{\frac{1}{1-\alpha } } -
\theta ^{\frac{1}{1-\alpha } }\right ] \to ^{(d)} N(0,
\sigma _{\frac{1}{q}, q} ^{2})
\end{equation*}
which is equivalent to (\ref{5m-2}).
\end{proof}

\vskip0.2cm

Using the so-called delta-method, we can get the asymptotic behavior of
the estimator $\widehat{\theta }_{n}$. Recall that if $(X_{n})_{n
\geq 1} $ is a sequence of random variables such that
\begin{equation*}
\sqrt{n} (X_{n}-\gamma _{0} ) \to ^{(d)} N (0, \sigma ^{2})
\end{equation*}
and $g$ is a function such that $g'(\gamma _{0} ) $ exists and does not
vanish, then
%
\begin{equation}
\label{22m-5}
\sqrt{n} ( g(X_{n}) - g(\gamma _{0} )) \to ^{(d)} N (0, \sigma ^{2}g'(
\gamma _{0} ) ^{2}).
\end{equation}

\begin{prop}
\label{pp5}
Consider the estimator (\ref{hatt}) and let $s_{1, \theta , \alpha } $
be given by (\ref{5m-2}). Then as $n\to \infty $,
%
\begin{equation}
\label{5m-3}
\sqrt{n} (\widehat{\theta } _{n,1} -\theta )\to N (0, s_{1, \theta
, \alpha } ^{2} (1-\alpha ) ^{2} \theta ^{\frac{2\alpha }{\alpha -1} }),
\end{equation}
and for $n$ large enough,
\begin{equation*}
d_{W} \left ( \sqrt{n} (\widehat{\theta } _{n,1} -\theta ), N(0, s
_{1, \theta , \alpha } ^{2} (1-\alpha ) ^{2}
\theta ^{\frac{2\alpha }{\alpha -1} }) \right ) \leq c\frac{1}{
\sqrt{n}}.
\end{equation*}
\end{prop}
\begin{proof}
By applying the delta-method for the function $g(x)= x ^{1-\alpha } $, $X_{n} =\widehat{\theta }_{n,1} ^{\frac{1}{1-\alpha }}$ and $\gamma
_{0} =\theta ^{\frac{1}{1-\alpha }}$, we immediately obtain the
convergence (\ref{5m-3}). Concerning the rate of convergence, we can
write, with $\widetilde{\gamma _{0}} $ a random point located between
$X_{n}$ and $\gamma _{0}$,
\begin{eqnarray*}
\sqrt{n} (g(X_{n})- g(\gamma _{0}))
&=& \sqrt{n} g'(
\widetilde{\gamma _{0}} ) ( X_{n} - \gamma _{0} )
\\
&=& g'(\gamma _{0}) \sqrt{n} ( X_{n} - \gamma _{0} ) + \sqrt{n} ( X
_{n} - \gamma _{0} )( g'(\widetilde{\gamma _{0}} ) - g'(\gamma _{0}) )
\\
&=:&g'(\gamma _{0}) \sqrt{n} ( X_{n} - \gamma _{0} ) + T_{n}.
\end{eqnarray*}
We have, for $n$ large,
\begin{eqnarray*}
\mathbf{E} \vert T_{n}\vert
&=& \mathbf{E} \left | \sqrt{n} ( X_{n}
- \gamma _{0} )( g'(\widetilde{\gamma _{0}} ) - g'(\gamma _{0}) )\right |
\\
&\leq & \left ( \mathbf{E} \left ( \sqrt{n} ( X_{n} - \gamma _{0}
)\right ) ^{2} \right ) ^{\frac{1}{2}} \left ( \mathbf{E} ( g'(
\widetilde{\gamma _{0}} ) - g'(\gamma _{0}) )^{2} \right ) ^{
\frac{1}{2}}
\\
&\leq & c \left ( \mathbf{E} ( g'(\widetilde{\gamma _{0}} ) - g'(
\gamma _{0}) )^{2} \right ) ^{\frac{1}{2}} \leq c \left ( \mathbf{E}
\left ( \widehat{\theta } _{n,1} ^{\frac{\alpha }{\alpha -1}} -
\theta ^{\frac{\alpha }{\alpha -1}}\right ) ^{2} \right ) ^{
\frac{1}{2}}
\\
&\leq & c \left ( \mathbf{E} \left ( \widehat{\theta } _{n,1} ^{\frac{1}{
\alpha -1}} - \theta ^{\frac{1}{\alpha -1}}\right ) ^{2} \right ) ^{
\frac{1}{2}} \leq c \frac{1}{\sqrt{n}}
\end{eqnarray*}
where we used the assumption $\alpha >1$ for the first inequality of the
line above and relation (\ref{5m-2}) (which gives in particular the
$L^{2}(\Omega )$-convergence of $ \widehat{\theta } _{n,1} ^{\frac{1}{
\alpha -1}}$ to $\theta ^{\frac{\alpha }{\alpha -1}}$ as $n\to \infty
$) for the second inequality on the same line. Therefore, by the
triangle inequality and Proposition \ref{pp4}, for $n$ large enough,
\begin{align*}
&d_{W} \left ( \sqrt{n} (\widehat{\theta } _{n,1} -\theta ), N(0, s
_{1, \theta , \alpha } ^{2} (1-\alpha ) ^{2}
\theta ^{\frac{\alpha }{\alpha -1} }) \right )
\\
& \leq cd_{W} \left ( \sqrt{n}(X_{n}-\gamma _{0}), N(0, s_{1,
\theta , \alpha } ^{2}) \right )+ \mathbf{E} \vert T_{n}\vert \leq c\frac{1}{
\sqrt{n}}.\qedhere
\end{align*}
\end{proof}

\subsubsection{Estimators based on the spatial variation}%
\label{sec2.3.3}
It is possible to define an estimator for the parameter $\theta $ based
on the spatial variations of the solution (\ref{sol1}). The result in
Proposition \ref{pp3} says that the process $\left ( u_{\theta }(t,x),
x\in \mathbb{R}\right )$ is a perturbed fBm,\index{perturbed fBm} so we know its exact
variation in space. Below $x_{j}= A_{1}+ \frac{j}{n} (A_{2}-A_{1})$, $j=0,\ldots, n $, will denote a partition of the interval $[A_{1}, A_{2}]$.

\begin{prop}
\label{ppp3}
Let $u_{\theta }$ be given by (\ref{1}). Then
\begin{equation*}
\sum _{i=0} ^{n-1} \left | u_{\theta } (t,x_{j+1})-u_{\theta } (t,x
_{j})\right | ^{\frac{2}{\alpha -1} } \to _{n\to \infty } m_{ \alpha }
^{\frac{2}{\alpha -1}}\mu _{\frac{2}{\alpha -1}} (A_{2}-A_{1}) \vert
\theta \vert ^{\frac{-1}{\alpha -1}}
\end{equation*}
and if $q:=\frac{2}{\alpha -1}$ is an integer,
\begin{eqnarray*}
&&\frac{1}{\sqrt{n}}\sum _{i=0} ^{n-1} \left [ \left ( \frac{ n}{m
_{\alpha }^{\frac{2}{\alpha -1}}(A_{2}-A_{1})} \right )
\theta ^{\frac{1}{\alpha -1}}\left ( u_{\theta } ( t,x_{i+1}) -u_{
\theta } ( t,x_{i}) \right ) ^{\frac{2}{\alpha -1}} -
\mu _{\frac{2}{\alpha -1}} \right ]
\\
&&\to ^{(d)} N(0, \sigma ^{2} _{\frac{\alpha -1}{2},
\frac{2}{\alpha -1}} ).
\end{eqnarray*}

\end{prop}

Proposition \ref{ppp3} leads to the definition of the
estimator
%
\begin{equation}
\label{hatt2}
\widehat{\theta }_{n, 2}=
\left [(m_{ \alpha } ^{\frac{2}{\alpha -1}}
\mu _{\frac{2}{\alpha -1}} (A_{2}-A_{1}))^{-1} \sum _{i=0} ^{n-1} \left |
u_{\theta } (t,x_{j+1})-u_{\theta } (t,x_{j})\right | ^{\frac{2}{
\alpha -1} } \right ]^{1-\alpha },
\end{equation}
and we can immediately deduce from Proposition \ref{pp3} its asymptotic
proprieties.

\begin{prop}
The estimator (\ref{hatt2}) converges in probability as $n\to \infty
$ to the parameter $\theta $. Moreover, if $q:=\frac{2}{\alpha -1}$ is
an even integer,
%
\begin{equation}
\label{23m-1}
\sqrt{n} \left [ \widehat{\theta }_{n,2} ^{\frac{1}{1-\alpha } } -
\theta ^{\frac{1}{1-\alpha }} \right ] \to ^{(d)} N(0, s_{2, \theta ,
\alpha } ^{2}) \mbox{ with } s_{2, \theta , \alpha } ^{2} = \sigma
^{2} _{\frac{\alpha -1}{2}, \frac{2}{\alpha -1}}
\mu _{\frac{2}{\alpha -1} } ^{-2} \theta ^{\frac{2}{1-\alpha } }.
\end{equation}
Moreover, for $n$ large,
\begin{equation*}
d_{W} \left ( \sqrt{n} \left [ \widehat{\theta }_{n,2} ^{\frac{1}{1-
\alpha } } -\theta ^{\frac{1}{1-\alpha }} \right ] , N(0, s_{2, \theta
, \alpha } ^{2}) \right ) \leq c\frac{1}{\sqrt{n}}.
\end{equation*}
\end{prop}
\begin{proof}
Using the law of the process $\left ( u_{\theta } (t, x), x\in
\mathbb{R}\right ) $ obtained in Proposition \ref{pp3}, we deduce that
the Gaussian process\index{Gaussian process} $\left ( \theta ^{\frac{1}{2}} m_{\alpha } ^{-1} u
_{\theta }(t, x), x\in \mathbb{R} \right )$ is a perturbed fractional
Brownian motion.\index{fractional ! Brownian motion} Therefore, by relation (\ref{c11}) in Lemma \ref{l1},
\begin{eqnarray*}
&&\frac{1}{\sqrt{n}} \sum _{i=0} ^{n-1} \left ( \frac{n
\theta ^{\frac{1}{\alpha -1}}}{(A_{2}-A_{1}) m_{\alpha } ^{\frac{2}{
\alpha -1} } }\left ( u_{\theta } (t,x_{j+1})-u_{\theta }
(t,x_{j})\right ) ^{\frac{2}{\alpha -1} }
-\mu _{\frac{2}{\alpha -1}}\right )
\\
&=&\sqrt{n} \mu _{\frac{2}{\alpha -1}} \theta ^{\frac{1}{\alpha -1}}
\left [ \widehat{\theta }_{n,2} ^{\frac{1}{1-\alpha } } -
\theta ^{\frac{1}{1-\alpha }} \right ] \to ^{(d)}
_{n\to \infty }N\left (0, \sigma ^{2} _{\frac{\alpha -1}{2}, \frac{2}{
\alpha -1}}\right ).
\end{eqnarray*}
Moreover, Proposition \ref{pp1} implies that
\begin{equation*}
d_{W} \left ( \sqrt{n} \mu _{\frac{2}{\alpha -1}}
\theta ^{\frac{1}{\alpha -1}} \left [ \widehat{\theta }_{n,2} ^{\frac{1}{1-
\alpha } } -\theta ^{\frac{1}{1-\alpha }} \right ] , N(0, \sigma ^{2}
_{\frac{\alpha -1}{2}, \frac{2}{\alpha -1}} )\right ) \leq c\frac{1}{
\sqrt{n}}
\end{equation*}
and this obviously leads to the desired conclusion.
\end{proof}

By using the delta-method, we can obtain the asymptotic distribution of
$\widehat{\theta }_{n,2}$.

\begin{prop}
Let $\widehat{\theta }_{n,2}$ be given by (\ref{hatt2}). Then, with
$s_{2, \theta , \alpha }$ from (\ref{23m-1}), as $n\to \infty $,
\begin{equation*}
\sqrt{n} (\widehat{\theta } _{n,2} -\theta )\to ^{(d)} N \left (0, s
_{2, \theta , \alpha } ^{2} (1-\alpha ) ^{2}
\theta ^{\frac{2\alpha }{\alpha -1} }\right ),
\end{equation*}
and for $n$ large enough,
\begin{equation*}
d_{W} \left ( \sqrt{n} (\widehat{\theta } _{n,2} -\theta ), N(0, s
_{2, \theta , \alpha } ^{2} (1-\alpha ) ^{2}
\theta ^{\frac{2\alpha }{\alpha -1} }) \right ) \leq c\frac{1}{
\sqrt{n}}.
\end{equation*}
\end{prop}
\begin{proof}
It suffices to apply (\ref{22m-5}) to the function $g(x)= x ^{1-
\alpha }$ and $\gamma _{0}= \theta ^{\frac{1}{1-\alpha }}$ and to follow
the proof of Proposition \ref{pp5}.
\end{proof}

\begin{remark}\mbox{}

\begin{itemize}
\item
The estimators (\ref{hatt}) and (\ref{hatt2}) coincide with the
estimators in \cite{PoTr} in the case of the standard Laplacian\index{standard Laplacian}
$\alpha =2$.
\item
The distance of the estimators (\ref{hatt}) and (\ref{hatt2}) to their
limit distribution is of the same order, although they involve
$q$-variations with different $q$.
\end{itemize}
\end{remark}

\section{Heat equation\index{heat equation} with the fractional Laplacian\index{fractional ! Laplacian} and a white-colored noise}%
\label{sec3}
In this section, we will consider the stochastic heat equation\index{stochastic heat equation} with an
additive Gaussian noise which behaves as a Wiener process in time and
as a fractional Brownian motion\index{fractional ! Brownian motion} in space, i.e. its spatial covariance
is given by the so-called Riesz kernel.\index{Riesz kernel} We will again study the
distribution of the solution, its connection with the fractional and
bifractional Brownian motion\index{bifractional Brownian motion} and we apply the $q$-variation method to
obtain an asymptotically normal estimator for the drift parameter.

\subsection{General properties of the solution}%
\label{sec3.1}
We will consider the stochastic heat equation\index{stochastic heat equation}
%
\begin{equation}
\label{jada}
\frac{\partial }{\partial t} u_{\theta }(t, \mathbf{x}) = -\theta (-
\Delta ) ^{\frac{\alpha }{2}}u_{\theta }(t,\mathbf{x}) + \dot{W} ^{
\gamma } (t,\mathbf{x}),
\hskip0.5cm t\geq 0, \mathbf{x}\in \mathbb{R} ^{d}\xch{,}{.}
\end{equation}
with $u_{\theta }(0,\mathbf{x})=0$ for every $\mathbf{x}\in
\mathbb{R} ^{d}$. In (\ref{jada}), $ -(-\Delta ) ^{\frac{\alpha }{2}}$
denotes the fractional Laplacian\index{fractional ! Laplacian} with exponent $\frac{\alpha }{2}$,
$\alpha \in (1,2]$, and $W^{\gamma }$ is the so-called white-colored
noise, i.e. $W^{\gamma } (t,A), t\geq 0$, $A \in \mathcal{B}(\mathbb{R}
^{d}) $, is a centered Gaussian field with covariance
%
\begin{equation}
\label{covwc}
\mathbf{E} W^{\gamma } (t,A)W ^{\gamma } (s,B)=( t\wedge s)\int _{A}
\int _{B} f(\mathbf{x}-\mathbf{y}) d\mathbf{x} d\mathbf{y},
\end{equation}
where $f$ is the so-called Riesz kernel\index{Riesz kernel} of order $\gamma $ given by
%
\begin{equation}
\label{riesz}
f(\mathbf{x})=R_{\gamma }(\mathbf{x}):=g_{\gamma ,d}\Vert {\mathbf{x}}
\Vert ^{-d+\gamma }, \quad 0<\gamma <d,
\end{equation}
where $g_{\gamma ,d}=2^{d-\gamma } \pi ^{d/2} \Gamma ((d-\gamma )/2)/
\Gamma (\gamma /2)$. As usual, the mild solution\index{mild solution} to (\ref{jada}) is
given by
%
\begin{equation}
\label{sol2}
u_{\theta }(t, \mathbf{x})= \int _{0} ^{t} \int _{ \mathbb{R} ^{d}}G
_{\alpha } (\theta (t-s), \mathbf{x}-\mathbf{z}) W ^{\gamma } (ds, d
\mathbf{z}),
\end{equation}
where the above integral $W^{\gamma } (ds, d\mathbf{z})$ is a Wiener
integral\index{Wiener integral} with respect to the Gaussian noise\index{Gaussian noise} $W^{\gamma } $.

We know the following facts concerning the mild solution\index{mild solution} (\ref{sol2})
when $\theta =1$.
\begin{itemize}%
\item
The mild solution\index{mild solution} (\ref{jada}) is well-defined as a square integrable
process satisfying
\begin{equation*}
\sup _{t \in [0, T], \mathbf{x}\in \mathbb{R} ^{d} } \mathbf{E} \vert
u_{1}(t, \mathbf{x}) \vert ^{2} <\infty
\end{equation*}
if and only if
%
\begin{equation}
\label{22m-1}
d< \gamma + \alpha .
\end{equation}
In particular, condition (\ref{22m-1}) shows that the solution exists
in any spatial dimension $d$, via suitable choice of the parameter
$\gamma $.
\item
Assume (\ref{22m-1}) is satisfied. Then for every $\mathbf{x}\in
\mathbb{R}^{d}$, we have the following equivalence in distribution
%
\begin{equation}
\label{22m2}
\left ( u_{1}(t,\mathbf{x}), t\geq 0\right )\equiv ^{(d)} \left ( c
_{2, \alpha ,\gamma } B_{t} ^{\frac{1}{2}, 1-\frac{d-\gamma }{\alpha
}}, t\geq 0 \right ) ,
\end{equation}
where $B ^{\frac{1}{2}, 1-\frac{d-\gamma }{\alpha }}$ is a bifractional
Brownian motion\index{bifractional Brownian motion} with the Hurst parameters $H=\frac{1}{2}$\index{Hurst parameters} and $K=1-\frac{d-
\gamma }{\alpha }$ and
%
\begin{equation}
\label{ze2}
c_{2, \alpha ,\gamma } ^{2}= c_{1, \alpha ,\gamma } 2 ^{1-\frac{d-
\gamma }{\alpha }}
\end{equation}
with
\begin{equation*}
c_{1, \alpha , \gamma }= (2\pi ) ^{-d} \int _{\mathbb{R } ^{d}} d
\xi \Vert \xi \Vert ^{-\gamma } e ^{- \Vert \xi \Vert ^{\alpha }
}\frac{1}{2 (1-\frac{d-\gamma }{\alpha })}.
\end{equation*}
\item
For every $t\geq 0$, we have (see Proposition 4.6 in
\cite{MahTu})
%
\begin{equation}
\label{2k-1}
\left ( u(t,\mathbf{x}), \mathbf{x}\in \mathbb{R}^{d}\right ) \equiv
^{(d)} \left ( m_{\alpha ,\gamma } B ^{ \frac{\alpha +\gamma -d}{2}} (
\mathbf{x}) + S_{t}(\mathbf{x}), \mathbf{x}\in \mathbb{R}^{d} \right ),
\end{equation}
where $B^{\frac{\alpha +\gamma -d}{2}}$ is an isotropic $d$-dimensional
fractional Brownian motion\index{fractional ! Brownian motion} (see the next section) with the Hurst parameter
$\frac{\alpha +\gamma -d}{2}\ $, $(S_{t}(\mathbf{x}))_{\mathbf{x}
\in \mathbb{R}^{d}} $ is a centered Gaussian process with $C^{\infty
} $ sample paths and $m_{\alpha , \gamma }^{2}$ is an explicit numerical
constant.

\end{itemize}

\subsection{Perturbed isotropic fractional Brownian motion\index{fractional ! Brownian motion}}%
\label{sec3.2}
Since the law of the solution (\ref{sol2}) is related to the isotropic
fBm,\index{isotropic fBm} let us recall the definition of this process. The isotropic
$d$-parameter fBm (also known as the L\'{e}vy fBm) $(B_{d}^{ H}(
\mathbf{x}),\mathbf{x} \in \mathbb{R} ^{d}) $ with the Hurst parameter
$H\in (0,1)$ is defined as a centered Gaussian process, starting from
zero, with covariance function
%
\begin{equation}
\label{cov-iso}
\mathbf{E} ( B_{d} ^{H}(\mathbf{x} )B_{d}^{H}(\mathbf{y}))=
\frac{1}{2} \left ( \Vert {\mathbf{x}}\Vert ^{2H} + \Vert {\mathbf{y}}
\Vert ^{2H} -\Vert {\mathbf{x}}-\mathbf{y}\Vert ^{2H} \right ) \quad \mbox{ for
every } {\mathbf{x}}, \mathbf{y}\in \mathbb{R} ^{d},
\end{equation}
where $\Vert \cdot \Vert $ denotes the Euclidean norm in $\mathbb{R}
^{d}$. It can be also represented as a Wiener integral\index{Wiener integral} with respect to
the Wiener sheet, see \cite{He,Li}.

As in the one-parameter case, we define the $q$-variation of the
isotropic fBm\index{isotropic fBm} as the limit in probability as $n\to \infty $ of the
sequence
\begin{equation*}
S^{n, q} _{[A_{1}, A_{2}]} (B ^{H} )= \sum _{i=0} ^{n-1} \left | B_{d}
^{H} (\mathbf{x}_{i+1})- B _{d}^{H} (\mathbf{x} _{i})\right | ^{q},
\end{equation*}
where $\mathbf{x}_{i}= (x_{i} ^{(1)},\ldots , x_{i} ^{(d)}) $ with
$x_{i}^{(j)}=A_{1} + \frac{i}{n} (A_{2}-A_{1}) $ for $i=0,\ldots ,n$ and
$j=1,\ldots ,d$. And from \cite{MahTu} we know that the isotropic fBm
$(B^{H} (\mathbf{x}))_{\mathbf{x} \in \mathbb{R} ^{d}} $\index{isotropic fBm} has $\frac{1}{H}$-variation over $[A_{1}, A_{2}]$
which is equal to
\begin{equation*}
(A_{2}- A_{1} ) \mathbf{E} \vert B^{H} _{d}(\mathbf{1}) \vert ^{1/H}=(A
_{2}-A_{1}) \sqrt{d} \mathbf{E} \vert Z\vert ^{1/H}.
\end{equation*}

The $q$-variation of the isotropic fBm perturbed by a regular
multiparameter process has been obtained in \cite{MahTu}, Lemma
4.1.

\begin{lemma}
\label{l21}
Let $(B^{H} (\mathbf{x}))_{\mathbf{x} \in \mathbb{R} ^{d}} $ be a
$d$-parameter isotropic fBm and consider
a $d$-pa\-ram\-eter stochastic process $(X(\mathbf{x}))_{\mathbf{x}
\in \mathbb{R} ^{d}}$, independent
of $ B^{H}$, that satisfies
%
\begin{equation}
\label{26ss-2}
\mathbf{E} \xch{\big|}{(} X(\mathbf{x})- X(\mathbf{y}) \big| ^{2} \leq C
\Vert {\mathbf{x}} -\mathbf{y}\Vert ^{2}, \quad \mbox{ for every }
{\mathbf{x}}, \mathbf{y}\in \mathbb{R} ^{d}.
\end{equation}
Define
\begin{equation*}
Y(\mathbf{x})= B_{d} ^{H} (\mathbf{x})+ X(\mathbf{x}) \quad \mbox{ for every
} {\mathbf{x}} \in \mathbb{R} ^{d}.
\end{equation*}

Then:
\begin{enumerate}%
\item
The process $(Y(\mathbf{x}))_{\mathbf{x}\in \mathbb{R} ^{d}}$ has
$\frac{1}{H} $-variation which is equal to
\begin{equation*}
(A_{2}-A_{1}) \sqrt{d} \mathbf{E} \vert Z\vert ^{1/H}.
\end{equation*}
\item
If $H\in (0, \frac{1}{2})$ and $q\geq 2$,
%
\begin{align}
\label{c1111}
&\frac{1}{\sqrt{n}}V_{q,n}(Y ^{H} ):=\frac{1}{\sqrt{n}}\sum _{i=0}
^{n-1} \left [ \frac{n ^{ Hq}d ^{-Hq/2}}{ (A_{2}-A_{1}) ^{qH}} ( Y
^{H} (\mathbf{x}_{i+1} )- Y ^{H} (\mathbf{x}_{i})) ^{q} - \mu _{q}
\right ] \nonumber\\
&\to ^{(d)} N (0, \sigma _{ H, q } ^{2}).
\end{align}

\end{enumerate}
\end{lemma}

It is immediate to deduce the rate of convergence in the above central
limit theorem. Recall that we denoted by $d_{W}$ the Wasserstein
distance.\index{Wasserstein distance}

\begin{prop}
\label{pp8}
Let $Y^{H}$ be as in the statement of Lemma \ref{l21}. Then for $n$ large,
\begin{equation*}
d_{W} \left ( \frac{1}{\sqrt{n}}V_{q,n}(Y ^{H} ), N (0, \sigma _{ H,
q } ^{2})\right ) \leq C \frac{1}{\sqrt{n}}.
\end{equation*}
\end{prop}
\begin{proof}
We notice that the Gaussian vector $\left ( B_{d}^{H}
(\mathbf{x}_{i+1})- B_{d} ^{H} (\mathbf{x}_{i}) \right )_{ 0,1,\ldots,n-1}$
has the same law as $d^{H/2} (B^{H}({ x}_{j+1})-B^{H} ({{ x}_{j}}))
_{{ } 0,1,\ldots,n-1}$ where $B$ is a one-parameter fBm with the Hurst parameter
$H$ and we then apply Lemma \ref{l1}. Therefore, the distribution of the
sequence $\frac{1}{\sqrt{n}}V_{q,n}(B_{d} ^{H} ) $ is independent of
$d\geq 1$ and we can use the same argument as in the proof of
Proposition \ref{pp1} above.
\end{proof}

\subsection{Estimators for the drift paramater}%
\label{sec3.3}
Throughout this section we will assume (\ref{22m-1}). As in the previous
section, we will construct and analyze estimators for the drift
parameter $\theta $ by using the limit behavior of the variations (in
time and in space) of the process (\ref{sol2}).

\subsubsection{The law of the solution}%
\label{sec3.3.1}
Let us start by analyzing the distribution of the solution to
(\ref{jada}) and its link with the (bi)fractional Brownian motion.\index{fractional ! Brownian motion}
%
\begin{prop}
\label{pp9}
For every $\mathbf{x} \in \mathbb{R}^{d}$ and $\theta >0$, we have
\begin{equation*}
\left ( u_{\theta } (t,\mathbf{x}), t\geq 0\right ) \equiv ^{(d)}
\left ( \theta ^{-\frac{d-\gamma }{2\alpha }} c_{2, \alpha ,\gamma } B
_{t} ^{\frac{1}{2}, 1-\frac{d-\gamma }{\alpha }}, t\geq 0\right ),
\end{equation*}
where $B ^{\frac{1}{2}, 1-\frac{d-\gamma }{\alpha }} $ is a bifractional
Brownian motion\index{bifractional Brownian motion} with parameters $H=\frac{1}{2} $ and $K= 1-\frac{d-
\gamma }{\alpha }$ and the constant $c_{2, \alpha ,\gamma } $ is defined
by (\ref{ze2}).

\end{prop}
\begin{proof}
Denote
\begin{equation*}
v_{\theta }(t,\mathbf{x})= u_{\theta } \left ( \frac{t}{\theta },
\mathbf{x}\right )\quad  \mbox{ for every } t\geq 0, \mathbf{x}\in
\mathbb{R} ^{d}.
\end{equation*}
Then, as in Lemma \ref{ll1}, $v_{\theta }$ solves the equation
%
\begin{equation}
\label{3}
\frac{\partial v_{\theta }}{\partial t} (t,\mathbf{x})= - (-\Delta )
^{\frac{\alpha }{2}}v_{\theta }(t,\mathbf{x})+(\theta ) ^{-
\frac{1}{2}} \dot{\widetilde{W}^{\gamma }} (t,\mathbf{x}),
\hskip0.4cm t\geq 0,\mathbf{x}\in \mathbb{R}^{d},
\end{equation}
with $ v_{\theta } (0, \mathbf{x})=0$ for every $\mathbf{x}\in
\mathbb{R}^{d}$, where $\dot{\widetilde{W}^{\gamma }}$ is a white
colored Gaussian noise\index{Gaussian noise} (i.e. a Gaussian process\index{Gaussian process} with zero mean and
covariance (\ref{covwc})).

Fix $\mathbf{x}\in \mathbb{R}^{d}$ and $\theta >0$.
For every $s, t\geq 0$, we have
\begin{align*}
\mathbf{E} u_{\theta } (t, \mathbf{x}) u_{\theta } (s,\mathbf{x})
&=
\mathbf{E} v_{\theta } (\theta t, \mathbf{x}) v_{\theta }(\theta s,
\mathbf{x})
\\
&= \theta ^{-1} \mathbf{E} u_{1} (\theta t, \mathbf{x}) u_{1} (\theta
s, \mathbf{x})
\\
&=\theta ^{-1} c_{1, \alpha ,\gamma } \left [ (\theta t+\theta s)
^{1-\frac{d-\gamma }{\alpha }} -\vert \theta t -\theta s
\vert ^{1-\frac{d-\gamma }{\alpha }} \right ]
\\
&= \theta ^{-\frac{d-\gamma }{\alpha }} c_{2, \alpha ,\gamma } ^{2}
\mathbf{E} B ^{\frac{1}{2}, 1-\frac{d-\gamma }{\alpha }} _{t} B ^{
\frac{1}{2}, 1-\frac{d-\gamma }{\alpha }} _{s}.\qedhere
\end{align*}

\end{proof}

For the behavior with respect to the space variable, we obtain the
following result.
%
\begin{prop}
\label{pp10}
For every $t\geq 0$, $\theta >0$, we have the following equality in
distribution
\begin{equation*}
\left ( u_{\theta }(t,\mathbf{x}),\mathbf{x}\in \mathbb{R}^{d}\right )
\equiv ^{(d)} \left ( \theta ^{-\frac{1}{2}} m_{\alpha ,\gamma } B ^{ \frac{
\alpha +\gamma -d}{2}} (\mathbf{x}) + S_{\theta t}(\mathbf{x}),
\mathbf{x}\in \mathbb{R}^{d} \right ),
\end{equation*}
where $B^{\frac{\alpha +\gamma --d}{2}}$ is a fractional Brownian motion\index{fractional ! Brownian motion}
with the Hurst parameter $\frac{\alpha +\gamma -d}{2}\in (0, \frac{1}{2}]
$, $(S_{\theta t}(\mathbf{x}))_{\mathbf{x}\in \mathbb{R}^{d}} $ is a
centered Gaussian process with $C^{\infty } $ sample paths and
$m_{\alpha , \gamma }$ from (\ref{2k-1}).

\end{prop}
\begin{proof}
The result is immediate since for a fixed time $t>0$
\begin{equation*}
\left ( u_{\theta }(t,\mathbf{x}), \mathbf{x}\in \mathbb{R}^{d}\right )=
\left ( v_{\theta }(\theta t,\mathbf{x}), \mathbf{x}\in \mathbb{R}
^{d}\right )\equiv ^{(d)} \theta ^{-\frac{1}{2} }\left ( u_{1} (\theta
t, \mathbf{x}), \mathbf{x}\in \mathbb{R}^{d}\right )
\end{equation*}
\begin{equation*}
\equiv ^{(d)} \left ( \theta ^{-\frac{1}{2}} m_{\alpha ,\gamma } B ^{ \frac{
\alpha +\gamma -d}{2}} (\mathbf{x}) + S_{\theta t}(\mathbf{x}),
\mathbf{x}\in \mathbb{R}^{d} \right ).\qedhere
\end{equation*}
\end{proof}

\subsubsection{Estimators based on the temporal variation}%
\label{sec3.3.2}
Again $t_{j}= A_{1}+ \frac{j}{n} (A_{2}-A_{1})$, $j=0,\ldots , n $, will denote
a partition of the interval $[A_{1}, A_{2}]$.

\begin{lemma}
Assume (\ref{22m-1}). Let $u_{\theta } $ be the solution to
(\ref{jada}). Then for every $\mathbf{x}\in \mathbb{R}^{d}$, the process
$\left ( u_{\theta }(t,\mathbf{x}), t\geq 0\right ) $ admits
$\frac{2\alpha }{\alpha +\gamma -d}$-variation over the interval
$[A_{1} ,A_{2}]$, i.e.
%
\begin{eqnarray}
\label{28f-33}
&&S^{n, \frac{2\alpha }{\alpha +\gamma -d}} _{[A_{1}, A_{2}]} : =
\sum _{i=0}^{n-1} \left | u_{\theta }(t_{j+1}, \mathbf{x})- u_{\theta
} (t_{j}, \mathbf{x}) \right | ^{\frac{2\alpha }{\alpha +\gamma -d}}
\nonumber \\
&& \to _{n\to \infty }c_{2, \alpha ,\gamma } ^{\frac{2\alpha }{\alpha
+\gamma -d}}2 ^{\frac{d-\gamma }{\alpha +\gamma -d}}
\mu _{\frac{2\alpha }{\alpha +\gamma -d}} (A_{2}-A_{1}) \vert \theta
\vert ^{\frac{\gamma -d}{\alpha +\gamma -d}}
\end{eqnarray}
in probability.
\end{lemma}
\begin{proof}
Clearly, for fixed $\mathbf{x}\in \mathbb{R} ^{d}$,
\begin{equation*}
\sum _{i=0} ^{n-1} \left | u_{\theta } (t_{j+1}, \mathbf{x})-u_{\theta
}(t_{j}, \mathbf{x})\right | ^{\frac{2\alpha }{\alpha +\gamma -d} }=
\sum _{i=0} ^{n-1} \left | v(\theta t_{j+1},\mathbf{x})-v(\theta t_{j},
\mathbf{x})\right | ^{\frac{2\alpha }{\alpha +\gamma -d} },
\end{equation*}
where $(v_{\theta }(t,\mathbf{x}), t\geq 0) \equiv ^{(d)} (
\theta ^{-\frac{1}{2}}u_{1}(t,\mathbf{x}), t\geq 0)$. And from
Proposition $4.3$ in \cite{MahTu} we know that $u_{1} $ admits a
variation of order $\frac{2\alpha }{\alpha +\gamma -d}$ which is equal
to\break
$ c_{2,\alpha ,\gamma }^{\frac{2\alpha }{\alpha +\gamma -d}}C_{
\frac{1}{2},1-\frac{d-\gamma }{\alpha }}(A_{2}-A_{1})$ with
$C_{\frac{1}{2},1-\frac{d-\gamma }{\alpha }}=2^{\frac{d-\gamma }{
\alpha +\gamma -d}}\mu _{\frac{2\alpha }{\alpha +\gamma -d}}$ and it means
that
\begin{align*}
&\sum _{i=0} ^{n-1} \left | u_{\theta } (t_{j+1}, \mathbf{x})-u_{
\theta } (t_{j}, \mathbf{x})\right | ^{\frac{2\alpha }{\alpha +\gamma
-d} }
\\
&\to _{n\to \infty } c_{2, \alpha ,\gamma } ^{\frac{2\alpha }{\alpha +
\gamma -d}}2 ^{\frac{d-\gamma }{\alpha +\gamma -d}}
\mu _{\frac{2\alpha }{\alpha +\gamma -d}}(\theta A_{2}-\theta A_{1})
\vert \theta ^{-\frac{1}{2}} \vert ^{\frac{2\alpha }{\alpha +\gamma -d}}
\\
&c_{2, \alpha ,\gamma } ^{\frac{2\alpha }{\alpha +\gamma -d}}2 ^{\frac{d-
\gamma }{\alpha +\gamma -d}} \mu _{\frac{2\alpha }{\alpha +\gamma -d}}
(A_{2}-A_{1}) \vert \theta \vert ^{\frac{\gamma -d}{\alpha +\gamma -d}}.\qedhere
\end{align*}
\end{proof}

From relation (\ref{28f-33}) we can naturally define the following
estimator for the parameter $\theta >0$ of the stochastic partial
differential equation (\ref{jada})
%
\begin{eqnarray}
\label{hatt3}
\widehat{\theta }_{n,3}
&=& \left ( \left ( c_{2, \alpha ,\gamma }
^{\frac{2\alpha }{\alpha +\gamma -d}} 2^{\frac{d-\gamma }{\alpha +
\gamma -d}}\mu _{\frac{2\alpha }{\alpha +\gamma -d}}
(A_{2}-A_{1})\right ) ^{-1}\right .
\nonumber\\
&&\times \left . \sum _{i=0} ^{n-1} \left | u_{\theta }(t_{j+1},
\mathbf{x})-u_{\theta }(t_{j}, \mathbf{x})\right | ^{\frac{2\alpha }{
\alpha +\gamma -d} } \right ) ^{\frac{\alpha +\gamma -d}{\gamma -d}}
\nonumber
\\
&=& \ \left ( c_{2, \alpha ,\gamma } ^{\frac{2\alpha }{\alpha +\gamma
-d}} 2^{\frac{d-\gamma }{\alpha +\gamma -d}}
\mu _{\frac{2\alpha }{\alpha +\gamma -d}}(A_{2}-A_{1})\right )^{\frac{d-
\gamma }{\alpha +\gamma -d}}
\nonumber
\\
&&\times \left ( S^{n, \frac{2\alpha }{\alpha +\gamma -d}}(u_{\theta
}(\cdot , \mathbf{x}))\right ) ^{\frac{\alpha +\gamma -d}{\gamma -d}},
\end{eqnarray}
and so
%
\begin{equation}
\label{5n-11}
\widehat{\theta } _{n,3} ^{\frac{\gamma -d}{\alpha +\gamma -d}} =
\frac{1}{ c_{2, \alpha ,\gamma } ^{\frac{2\alpha }{\alpha +\gamma -d}}
2^{\frac{d-\gamma }{\alpha +\gamma -d}}
\mu _{\frac{2\alpha }{\alpha +\gamma -d}} (A_{2}-A_{1})} \sum _{i=0}
^{n-1} \left | u_{\theta }(t_{j+1},\mathbf{x})-u_{\theta }(t_{j},
\mathbf{x})\right | ^{\frac{2\alpha }{\alpha +\gamma -d} }.
\end{equation}

We have the following asymptotic behavior.

\begin{prop}
\label{pp22}
Assume $\frac{2\alpha }{\alpha +\gamma -d}:=q$ is an even integer and
consider the estimator $\widehat{\theta }_{n,3}$ in (\ref{hatt3}). Then
$\widehat{\theta }_{n,3} \to _{n \infty }\theta $ in probability and
%
\begin{equation}
\label{5m-22}
\sqrt{n} \left [ \widehat{\theta }_{n,3} ^{\frac{\gamma -d}{\alpha
+\gamma -d} } -\theta ^{\frac{\gamma -d}{\alpha +\gamma -d} }\right ]
\to ^{(d)} N(0, s_{3, \theta , \alpha , \gamma } ^{2}) \mbox{ with }s
_{3, \theta , \alpha , \gamma } ^{2} = \sigma _{\frac{1}{q}, q} ^{2}
\theta ^{\frac{2(\gamma -d)}{\alpha +\gamma -d} }
\mu _{\frac{2\alpha }{\alpha +\gamma -d} } ^{-2}\xch{,}{.}
\end{equation}
and for $n$ large enough,
%
\begin{equation}
\label{22m-4}
d_{W}\left ( \sqrt{n} \left [ \widehat{\theta }_{n,3} ^{\frac{
\gamma -d}{\alpha +\gamma -d} } -
\theta ^{\frac{\gamma -d}{\alpha +\gamma -d} }\right ] , N(0, s_{
\theta , \alpha } ^{2})\right ) \leq c\frac{1}{\sqrt{n}}.
\end{equation}
\end{prop}
\begin{proof}
From Proposition \ref{pp9} and the relation between the fractional and
bifractional Brownian motion\index{bifractional Brownian motion} (see (\ref{5m-1})), we can see that, as
$n\to \infty $,
\begin{equation*}
\left ( c_{2, \alpha , \gamma } ^{-1} 2 ^{\frac{d-\gamma }{2\alpha }}
\theta ^{\frac{d-\gamma }{2\alpha }} u_{\theta } (t,\mathbf{x}), t
\geq 0\right )
\end{equation*}
%
converges to a perturbed fBm\index{perturbed fBm} with Hurst parameter $H= \frac{\alpha -d+\gamma }{2
\alpha }$. By taking $H=\frac{\alpha +\gamma -d}{2\alpha }$ and $q=
\frac{1}{H} = \frac{2\alpha }{\alpha +\gamma -d } $ in Lemma \ref{l1},
we get
\begin{eqnarray*}
&&\frac{1}{\sqrt{n}} \sum _{i=0} ^{n-1} \left [\frac{ n
\theta ^{\frac{d-\gamma }{\alpha +\gamma -d}} }{ c_{2, \alpha ,\gamma
} ^{\frac{2\alpha }{\alpha +\gamma -d}} 2^{\frac{d-\gamma }{\alpha +
\gamma -d}} (A_{2}-A_{1}) }\left ( u_{\theta }(t_{j+1}, \mathbf{x})-u
_{\theta }(t_{j}, \mathbf{x})\right ) ^{\frac{2\alpha }{\alpha +
\gamma -d} }- \mu _{\frac{2\alpha }{\alpha +\gamma -d} }\right ]
\\
&&\to \xch{N(0, \sigma _{\frac{1}{q}, q} ^{2})}{N(0, \sigma _{\frac{1}{q}, q} ^{2}).}
\end{eqnarray*}
or, equivalently
\begin{equation*}
\sqrt{n} \mu _{\frac{2\alpha }{\alpha +\gamma -d} }
\theta ^{\frac{d-\gamma }{\alpha +\gamma -d} } \left [
\widehat{\theta }_{n,3} ^{\frac{\gamma -d}{\alpha +\gamma -d} } -
\theta ^{\frac{\gamma -d}{\alpha +\gamma -d} }\right ] \to N(0,
\sigma _{\frac{1}{q}, q} ^{2}),
\end{equation*}
which is equivalent to (\ref{5m-2}). The bound (\ref{22m-4}) follows
easily from Proposition \ref{pp1}.
\end{proof}

We finally obtain the asymptotic normality and the rate of convergence
for the estimator $\widehat{\theta }_{n,3}$.

\begin{prop}
Let $\widehat{\theta }_{n,3}$ be given by (\ref{hatt3}) and
$s_{3,\theta , \alpha , \gamma }$ be given by (\ref{5m-22}). Then as
$n\to \infty $,
\begin{equation*}
\sqrt{n} \left ( \widehat{\theta }_{n,3}- \theta \right ) \to ^{(d)}
N \left (0, s_{3,\theta , \alpha , \gamma } \left ( \frac{\alpha +
\gamma -d}{\gamma -d}\right ) ^{2}
\theta ^{\frac{2\alpha }{\alpha +\gamma - d}} \right )
\end{equation*}
and
\begin{equation*}
d_{W} \left ( \sqrt{n} \left ( \widehat{\theta }_{n,3}-
\theta \right ) , N \left (0, s_{3,\theta , \alpha , \gamma } \left (
\frac{\alpha +\gamma -d}{\gamma -d}\right ) ^{2}
\theta ^{\frac{2\alpha }{\alpha +\gamma - d}} \right )\right ) \leq c\frac{1}{
\sqrt{n}}.
\end{equation*}
\end{prop}
\begin{proof}
It suffices to apply (\ref{22m-5}) with $g(x)= x ^{\frac{\alpha +
\gamma -d}{\gamma -d}}$ and $\gamma _{0}=
\theta ^{\frac{\gamma -d}{\gamma +\alpha -d}}$ and to follow the proof
of Proposition \ref{pp5}.
\end{proof}

\subsection{Estimators based on the spatial variation}
\label{sec3.4}
We will repeat the method employed in the previous parts of our work in
order to define an estimator expressed in terms of the variations in
space of the process (\ref{sol2}) for the parameter $\theta $ in
(\ref{jada}).

Recall that we 
proved in Proposition \ref{pp10} that for every fixed time
$t>0$,
\begin{equation*}
\left ( \theta ^{\frac{1}{2} } m_{\alpha , \gamma }^{-1} u_{\theta } (t,
\mathbf{x}),\mathbf{x} \in \mathbb{R} ^{d}\right )
\end{equation*}
is a perturbed multiparameter isotropic fractional Brownian motion\index{fractional ! Brownian motion} as
defined in Lemma \ref{l21}. Then we can deduce the variation in space
of $u_{\theta }$ recalling that $\mathbf{x}_{i}= (x_{i} ^{(1)},\ldots ,
x_{i} ^{(d)}) $ with $x_{i}^{(j)}=A_{1} + \frac{i}{n} (A_{2}-A_{1}) $
for $i=0,\ldots ,n$ and $j=1,\ldots ,d$.

\begin{prop}
\label{mam}%
Let $u_{\theta } $ be given by (\ref{sol2}). Then
\begin{equation*}
\sum _{i=0} ^{n-1} \left | u_{\theta } (t,\mathbf{x}_{j+1})-u(_{\theta
}t,\mathbf{x}_{j})\right | ^{\frac{2}{\alpha +\gamma -d} }
\to _{n\to \infty } m_{\alpha ,\gamma }^{\frac{2}{\alpha +\gamma -d}}(A
_{2}-A_{1})\sqrt{d}\mu _{\frac{2}{\alpha +\gamma -d}}\left |\theta \right |
^{\frac{-1}{\alpha +\gamma -d}}
\end{equation*}

\end{prop}
\begin{proof}
We use Lemma \ref{l21}, point 1.
\end{proof}

\vskip0.2cm

For every $n\geq 1$, define
%
\begin{align}
\label{hatt4}
\widehat{\theta }_{n,4}&=
\Biggr[(m_{ \alpha ,\gamma } ^{\frac{2}{
\alpha +\gamma -d}}\mu _{\frac{2}{\alpha +\gamma -d}} \sqrt{d} (A
_{2}-A_{1}))^{-1}
\nonumber\\
&\quad \times \sum _{i=0} ^{n-1} \left | u_{\theta } (t,\mathbf{x}
_{j+1})-u_{\theta } (t,\mathbf{x}_{j})\right | ^{\frac{2}{\alpha +
\gamma -d} } \Biggl]^{-(\alpha +\gamma -d)},
\end{align}
and so
%
\begin{equation}
\label{5n-111}
\widehat{\theta } _{n,4} ^{\frac{-1}{\alpha +\gamma -d}} = \frac{1}{ m
_{ \alpha ,\gamma } ^{\frac{2}{\alpha +\gamma -d}}
\mu _{\frac{2}{\alpha +\gamma -d}} \sqrt{d}(A_{2}-A_{1})} \sum _{i=0}
^{n-1} \left | u_{\theta }(t,\mathbf{x}_{j+1})-u_{\theta }(t,
\mathbf{x}_{j})\right | ^{\frac{2}{\alpha +\gamma -d} }.
\end{equation}

We can deduce the asymptotic properties of the estimator by using Lemma
\ref{l21} and Proposition \ref{pp8}.

\begin{prop}
The estimator (\ref{hatt4}) converges in probability as $n\to \infty
$ to the parameter $\theta $. Moreover, if $ \frac{2}{\alpha + \gamma
-d}$ is an even integer, then
\begin{align*}
&\sqrt{n} \left [ \widehat{\theta }_{n,4} ^{\frac{-1}{\alpha +\gamma
-d }} -\theta ^{\frac{-1}{\alpha +\gamma -d }} \right ] \to N(0, s^{2}
_{4, \theta , \alpha , \gamma } )
\\
&\quad  \mbox{ with } s^{2}_{4, \theta ,
\alpha , \gamma } =\sigma ^{2} _{\frac{\alpha +\gamma -1}{2}, \frac{2}{
\alpha +\gamma -d}}\mu _{\frac{2}{\alpha +\gamma -d} } ^{-2}
\theta ^{\frac{-2}{\alpha +\gamma -d } }.
\end{align*}
We also have, for $n$ large enough,
\begin{equation*}
d_{W} \left ( \sqrt{n} \left [ \widehat{\theta }_{n,4} ^{\frac{-1}{
\alpha +\gamma -d }} -\theta ^{\frac{-1}{\alpha +\gamma -d }} \right ],
N(0, s^{2}_{4, \theta , \alpha , \gamma } )\right ) \leq c\frac{1}{
\sqrt{n}}.
\end{equation*}
\end{prop}

Finally, we get the following proposition.

\begin{prop}
With $\widehat{\theta } _{n,4}$ from (\ref{hatt4}), as $n\to \infty $,
\begin{equation*}
\sqrt{n} \left ( \widehat{\theta }_{n,4}- \theta \right ) \to ^{(d)}
N \left (0, s_{4,\theta , \alpha , \gamma } \left ( \frac{\alpha +
\gamma -d}{\gamma -d}\right ) ^{2}
\theta ^{\frac{2\alpha }{\alpha +\gamma - d}} \right )
\end{equation*}
and
\begin{equation*}
d_{W} \left ( \sqrt{n,4} \left ( \widehat{\theta }_{n}-
\theta \right ) , N \left (0, s_{4,\theta , \alpha , \gamma } \left (
\frac{\alpha +\gamma -d}{\gamma -d}\right ) ^{2}
\theta ^{\frac{2\alpha }{\alpha +\gamma - d}} \right )\right ) \leq c\frac{1}{
\sqrt{n}}.
\end{equation*}
\end{prop}
\begin{proof}
Apply again (\ref{22m-5}) with $g(x)= x ^{\frac{\alpha +\gamma -d}{
\gamma -d}}$ and $\gamma _{0}= \theta ^{\frac{\gamma -d}{\gamma +\alpha
-d}}$.
\end{proof}

\begin{remark}
Notice that in the case $\gamma =1$ (i.e., there is no spatial
correlation and in this case $d$ has to be 1), we retrieve the results of
Section \ref{sec2}. Observe, as in Section \ref{sec2}, that the distance
of the estimators (\ref{hatt3}) and (\ref{hatt4}) to their limit
distribution is of the same order, although they involve $q$-variations
of different orders.
\end{remark}

\section{Conclusion}\vspace*{-2pt}
\label{sec4}
To conclude, in 
this paper we provide estimators based on power variation
for the drift parameter $\theta $ of the solution to the fractional
stochastic heat equation\index{fractional ! stochastic heat equation}\index{stochastic heat equation fractional} (\ref{1}). The novelty of our approach is that
it allows, comparing with the literature on statistical inference for
SPDEs (see \cite{Cia,Ma,BT1}, etc.), to
consider the case of a Gaussian noise\index{Gaussian noise} with non-trivial spatial
correlation and to treat the situation when the differential operator
in the heat equation\index{heat equation} (\ref{1}) is the fractional Laplacian\index{fractional ! Laplacian} instead of
the standard Laplacian.\index{standard Laplacian} The proofs of the asymptotic behavior of the
estimators are relatively simple and they are based on the link between
the law of the solution and the fractional Brownian motion,\index{fractional ! Brownian motion} using known
results on the behavior of the power variations\index{power variations} of the fBm. Our approach
also gives the rate of convergence of the estimators under the
Wasserstein distance\index{Wasserstein distance} via some recent results in Stein--Malliavin calculus
(see \cite{NP-book}). We assumed for simplicity a vanishing
initial condition in (\ref{1}) but the case of a notrivial initial value,
whose power variations\index{power variations} are dominated by those of the fBm, can be also
treated by our approach. Another open problem of interest that could by
treated via our techniques is adding an unknown volatility parameter in
the disturbance term and jointly estimating the drift and the volatility
parameters. The case of the fractional heat equation\index{heat equation} on bounded domains
is also interesting but in this case the fundamental solution and
implicitly the law of the mild solution\index{mild solution} changes. Consequently, the
relation between the law of the solution and the fBm is not obvious and
therefore new techniques are needed.



\begin{funding}
C. Tudor is partially supported by \gsponsor[id=GS2]{Labex Cempi}
(\gnumber[refid=GS2]{ANR-11-LABX-0007-01}) and\break \gsponsor[id=GS3]{MATHAMSUD} Project SARC (\gnumber[refid=GS3]{19-MATH-06}).
\end{funding}


\begin{thebibliography}{99}\vspace*{-2pt}

\bibitem{BT2}
\begin{botherref}
\oauthor{\bsnm{Bibinger}, \binits{M.}},
\oauthor{\bsnm{Trabs}, \binits{M.}}:
On central limit theorems for power variations of the solution to the
 stochastic heat equation
\end{botherref}
%
\OrigBibText
\begin{botherref}
\oauthor{\bsnm{Bibinger}, \binits{M.}},
\oauthor{\bsnm{Trabs}, \binits{M.}}:
On central limit theorems for power variations of the solution to the
 stochastic heat equation
\end{botherref}
\endOrigBibText
\bptok{structpyb}%
\endbibitem

\bibitem{BT1}
\begin{botherref}
\oauthor{\bsnm{Bibinger}, \binits{M.}},
\oauthor{\bsnm{Trabs}, \binits{M.}}:
Volatility estimation for stochastic PDEs using high-frequency observations.
Stoch. Process. Appl., in press,
\bid{doi={10.1016/j.spa.2019.09.002}}
\end{botherref}
%
\OrigBibText
\begin{botherref}
\oauthor{\bsnm{Bibinger}, \binits{M.}},
\oauthor{\bsnm{Trabs}, \binits{M.}}:
Volatility estimation for stochastic PDEs using high-frequency observations
\end{botherref}
\endOrigBibText
\bptok{structpyb}%
\endbibitem

\bibitem{Chong}
\begin{botherref}
\oauthor{\bsnm{Chong}, \binits{C.}}:
High-frequency analysis of parabolic stochastic PDEs.
Ann. Stat., forthcoming
\end{botherref}
%
\OrigBibText
\begin{botherref}
\oauthor{\bsnm{Chong}, \binits{C.}}:
High-frequency analysis of parabolic stochastic pdes
\end{botherref}
\endOrigBibText
\bptok{structpyb}%
\endbibitem

\bibitem{Cia}
\begin{barticle}
\bauthor{\bsnm{Cialenco}, \binits{I.}}:
\batitle{Statistical inference for SPDEs: an overview}.
\bjtitle{Stat. Inference Stoch. Process.}
\bvolume{21}(\bissue{2}),
\bfpage{309}--\blpage{329}
(\byear{2018})
\bid{doi={10.1007/s11203-018-9177-9}, mr={3824970}}
\end{barticle}
%
\OrigBibText
\begin{barticle}
\bauthor{\bsnm{Cialenco}, \binits{I.}}:
\batitle{Statistical inference for spdes: an overview}.
\bjtitle{Statistical Inference for Stochastic Processes}
\bvolume{21(2)},
\bfpage{309}--\blpage{329}
(\byear{2018})
\end{barticle}
\endOrigBibText
\bptok{structpyb}%
\endbibitem

\bibitem{CH}
\begin{botherref}
\oauthor{\bsnm{Cialenco}, \binits{I.}},
\oauthor{\bsnm{Huang}, \binits{Y.}}:
A note on parameter estimation for discretely sampled SPDEs
\end{botherref}
%
\OrigBibText
\begin{botherref}
\oauthor{\bsnm{Cialenco}, \binits{I.}},
\oauthor{\bsnm{Huang}, \binits{Y.}}:
A note on parameter estimation for discretely sampled spdes
\end{botherref}
\endOrigBibText
\bptok{structpyb}%
\endbibitem

\bibitem{DD}
\begin{barticle}
\bauthor{\bsnm{Debbi}, \binits{L.}},
\bauthor{\bsnm{Dozzi}, \binits{M.}}:
\batitle{On the solutions of nonlinear stochastic fractional partial
 differential equations in one spatial dimension}.
\bjtitle{Stoch. Process. Appl.}
\bvolume{115},
\bfpage{1761}--\blpage{1781}
(\byear{2005})
\bid{doi={10.1016/j.spa.2005.06.001}, mr={2172885}}
\end{barticle}
%
\OrigBibText
\begin{barticle}
\bauthor{\bsnm{Debbi}, \binits{L.}},
\bauthor{\bsnm{Dozzi}, \binits{M.}}:
\batitle{On the solutions of nonlinear stochastic fractional partial
 differential equations in one spatial dimension}.
\bjtitle{Stochastic Processes and their Applications}
\bvolume{115},
\bfpage{1761}--\blpage{1781}
(\byear{2005})
\end{barticle}
\endOrigBibText
\bptok{structpyb}%
\endbibitem

\bibitem{FKM}
\begin{barticle}
\bauthor{\bsnm{Foondun}, \binits{M.}},
\bauthor{\bsnm{Khoshnevisan}, \binits{D.}},
\bauthor{\bsnm{Mahboubi}, \binits{P.}}:
\batitle{Analysis of the gradient of the solution to a stochastic heat equation
 via fractional Brownian motion}.
\bjtitle{Stoch. Partial Differ. Equ., Anal. Computat.}
\bvolume{3}(\bissue{2}),
\bfpage{133}--\blpage{158}
(\byear{2015})
\bid{doi={10.1007/s40072-015-0045-y}, mr={3350450}}
\end{barticle}
%
\OrigBibText
\begin{barticle}
\bauthor{\bsnm{Foondun}, \binits{M.}},
\bauthor{\bsnm{Khoshnevisan}, \binits{D.}},
\bauthor{\bsnm{Mahboubi}, \binits{P.}}:
\batitle{Analysis of the gradient of the solution to a stochastic heat equation
 via fractional brownian motion}.
\bjtitle{Stochastic Partial Differential Equations Analysis and Computation}
\bvolume{3(2)},
\bfpage{133}--\blpage{158}
(\byear{2015})
\end{barticle}
\endOrigBibText
\bptok{structpyb}%
\endbibitem

\bibitem{He}
\begin{barticle}
\bauthor{\bsnm{Herbin}, \binits{H.}}:
\batitle{From $n$-parameter fractional Brownian motion to $n$-parameter
 multifractional Brownian motion}.
\bjtitle{Rocky Mt. J. Math.}
\bvolume{36}(\bissue{4}),
\bfpage{1249}--\blpage{1284}
(\byear{2006})
\bid{doi={10.1216/rmjm/1181069415}, mr={2274895}}
\end{barticle}
%
\OrigBibText
\begin{barticle}
\bauthor{\bsnm{Herbin}, \binits{H.}}:
\batitle{From $n$-parameter fractional brownian motion to $n$-parameter
 multifractional brownian motion}.
\bjtitle{Rocky Mountain Journal of Mathematics}
\bvolume{36(4)},
\bfpage{1249}--\blpage{1284}
(\byear{2006})
\end{barticle}
\endOrigBibText
\bptok{structpyb}%
\endbibitem

\bibitem{HV}
\begin{barticle}
\bauthor{\bsnm{Houdr\'e}, \binits{C.}},
\bauthor{\bsnm{Villa}, \binits{J.}}:
\batitle{An example of infinite dimensional quasi-helix}.
\bjtitle{Stoch. Models, Contemp. Math.}
\bvolume{366},
\bfpage{195}--\blpage{201}
(\byear{2003})
\bid{doi={10.1090/conm/336/06034}, mr={2037165}}
\end{barticle}
%
\OrigBibText
\begin{barticle}
\bauthor{\bsnm{Houdr\'e}, \binits{C.}},
\bauthor{\bsnm{Villa}, \binits{J.}}:
\batitle{An example of infinite dimensional quasi-helix}.
\bjtitle{Stochastic Models, Contemporary Mathematics}
\bvolume{366},
\bfpage{195}--\blpage{201}
(\byear{2003})
\end{barticle}
\endOrigBibText
\bptok{structpyb}%
\endbibitem

\bibitem{Ja2}
\begin{barticle}
\bauthor{\bsnm{Jacob}, \binits{A.} \bsuffix{N.and~potrykus}},
\bauthor{\bsnm{Wu}, \binits{J.-L.}}:
\batitle{Solving a non-linear stochastic pseudo-differential equation of
 Burgers type}.
\bjtitle{Stoch. Process. Appl.}
\bvolume{120},
\bfpage{2447}--\blpage{2467}
(\byear{2010})
\bid{doi={10.1016/j.spa.2010.08.007}, mr={2728173}}
\end{barticle}
%
\OrigBibText
\begin{barticle}
\bauthor{\bsnm{Jacob}, \binits{A.} \bsuffix{N.and~potrykus}},
\bauthor{\bsnm{Wu}, \binits{J.-L.}}:
\batitle{Solving a non-linear stochastic pseudo-differential equation of
 burgers type}.
\bjtitle{Stochastic Processes and their Applicatiobs}
\bvolume{120},
\bfpage{2447}--\blpage{2467}
(\byear{2010})
\end{barticle}
\endOrigBibText
\bptok{structpyb}%
\endbibitem

\bibitem{Ja1}
\begin{barticle}
\bauthor{\bsnm{Jacob}, \binits{N.}},
\bauthor{\bsnm{Leopold}, \binits{G.}}:
\batitle{Pseudo differential operators with variable order of differentiation
 generating Feller semigroups}.
\bjtitle{Integral Equ. Oper. Theory}
\bvolume{17},
\bfpage{544}--\blpage{553}
(\byear{1993})
\bid{doi={10.1007/BF01200393}, mr={1243995}}
\end{barticle}
%
\OrigBibText
\begin{barticle}
\bauthor{\bsnm{Jacob}, \binits{N.}},
\bauthor{\bsnm{Leopold}, \binits{G.}}:
\batitle{Pseudo differential operators with variable order of differentiation
 generating feller semigroups}.
\bjtitle{Integral Equations Operator Theory}
\bvolume{17},
\bfpage{544}--\blpage{553}
(\byear{1993})
\end{barticle}
\endOrigBibText
\bptok{structpyb}%
\endbibitem

\bibitem{Jiang}
\begin{barticle}
\bauthor{\bsnm{Jiang}, \binits{Y.}},
\bauthor{\bsnm{shi}, \binits{K.}},
\bauthor{\bsnm{Wang}, \binits{Y.}}:
\batitle{Stochastic fractional Anderson models with fractional noises}.
\bjtitle{Chin. Ann. Math.}
\bvolume{31B}(\bissue{1}),
\bfpage{101}--\blpage{118}
(\byear{2010})
\bid{doi={10.1007/s11401-008-0244-1}, mr={2576182}}
\end{barticle}
%
\OrigBibText
\begin{barticle}
\bauthor{\bsnm{Jiang}, \binits{Y.}},
\bauthor{\bsnm{shi}, \binits{K.}},
\bauthor{\bsnm{Wang}, \binits{Y.}}:
\batitle{Stochastic fractional anderson models with fractional noises}.
\bjtitle{Chinese Annals of Mathematics}
\bvolume{31B(1)},
\bfpage{101}--\blpage{118}
(\byear{2010})
\end{barticle}
\endOrigBibText
\bptok{structpyb}%
\endbibitem

\bibitem{MahTu}
\begin{barticle}
\bauthor{\bsnm{Khalil-Mahdi}, \binits{Z.}},
\bauthor{\bsnm{Tudor}, \binits{C.}}:
\batitle{On the distribution and $q$-variation of the solution to the heat equation with
 fractional Laplacian}.
\bjtitle{Probab. Theory Math. Stat.}
\bvolume{39}(\bissue{2})
(\byear{2019})
\bid{doi={10.19195/0208-4147.39.2.5}}
\end{barticle}
%
\OrigBibText
\begin{botherref}
\oauthor{\bsnm{Khalil-Mahdi}, \binits{Z.}},
\oauthor{\bsnm{Tudor}, \binits{C.}}:
On the distribution and $q$-variation of the solution to the heat equation with
 fractional laplacian.
Probability Theory and Mathematical Statistics
(2018)
\end{botherref}
\endOrigBibText
\bptok{structpyb}%
\endbibitem

\bibitem{LN}
\begin{barticle}
\bauthor{\bsnm{Lei}, \binits{P.}},
\bauthor{\bsnm{Nualart}, \binits{D.}}:
\batitle{A decomposition of the bifractional Brownian motion and some
 applications}.
\bjtitle{Stat. Probab. Lett.}
\bvolume{79}(\bissue{5}),
\bfpage{619}--\blpage{624}
(\byear{2008})
\bid{doi={10.1016/j.spl.2008.10.009}, mr={2499385}}
\end{barticle}
%
\OrigBibText
\begin{barticle}
\bauthor{\bsnm{Lei}, \binits{D.} \bsuffix{P.and~Nualart}}:
\batitle{A decomposition of the bifractional brownian motion and some
 applications}.
\bjtitle{Statististics and Probability Letters}
\bvolume{79(5)},
\bfpage{619}--\blpage{624}
(\byear{20})
\end{barticle}
\endOrigBibText
\bptok{structpyb}%
\endbibitem

\bibitem{Li}
\begin{barticle}
\bauthor{\bsnm{Lindstrom}, \binits{T.}}:
\batitle{Fractional Brownian fields as integrals of white noise}.
\bjtitle{Bull. Lond. Math. Soc.}
\bvolume{25},
\bfpage{893}--\blpage{898}
(\byear{1993})
\bid{doi={10.1112/blms/25.1.83}, mr={1190370}}
\end{barticle}
%
\OrigBibText
\begin{barticle}
\bauthor{\bsnm{Lindstrom}, \binits{T.}}:
\batitle{Fractional brownian fields as integrals of white noise}.
\bjtitle{Bulletin of London Mathematical Society}
\bvolume{25},
\bfpage{893}--\blpage{898}
(\byear{1993})
\end{barticle}
\endOrigBibText
\bptok{structpyb}%
\endbibitem

\bibitem{Lot}
\begin{barticle}
\bauthor{\bsnm{Lototsky}, \binits{S.}}:
\batitle{Statistical inference for stochastic parabolic equations: a spectral
 approach}.
\bjtitle{Publ. Math.}
\bvolume{53}(\bissue{1}),
\bfpage{3}--\blpage{45}
(\byear{2009})
\bid{doi={10.5565/PUBLMAT\_53109\_01}, mr={2474113}}
\end{barticle}
%
\OrigBibText
\begin{barticle}
\bauthor{\bsnm{Lototsky}, \binits{S.}}:
\batitle{Statistical inference for stochastic parabolic equations: a spectral
 approach}.
\bjtitle{Publications Matematiques}
\bvolume{53(1)},
\bfpage{3}--\blpage{45}
(\byear{2009})
\end{barticle}
\endOrigBibText
\bptok{structpyb}%
\endbibitem

\bibitem{Ma}
\begin{barticle}
\bauthor{\bsnm{Markussen}, \binits{B.}}:
\batitle{Likelihood inference for a discretely observed stochastic partial
 differential equation}.
\bjtitle{Bernoulli}
\bvolume{9}(\bissue{5}),
\bfpage{745}--\blpage{762}
(\byear{2003})
\bid{doi={10.3150/bj/1066418876}, mr={2047684}}
\end{barticle}
%
\OrigBibText
\begin{barticle}
\bauthor{\bsnm{Markussen}, \binits{B.}}:
\batitle{Likelihood inference for a discretely observed stochastic partial
 differential equation}.
\bjtitle{Bernoull}
\bvolume{9(5)},
\bfpage{745}--\blpage{762}
(\byear{2003})
\end{barticle}
\endOrigBibText
\bptok{structpyb}%
\endbibitem

\bibitem{Mo}
\begin{barticle}
\bauthor{\bsnm{Mohapl}, \binits{J.}}:
\batitle{On estimation in the planar Ornstein--Ulenbeck process. Communications
 in statistics}.
\bjtitle{Stoch. Models}
\bvolume{13}(\bissue{3}),
\bfpage{435}--\blpage{455}
(\byear{1997})
\bid{doi={10.1080/15326349708807435}, mr={1457656}}
\end{barticle}
%
\OrigBibText
\begin{barticle}
\bauthor{\bsnm{Mohapl}, \binits{J.}}:
\batitle{On estimation in the planar ornstein-uhlenbeck process. communications
 in statistics}.
\bjtitle{Stochastic Models}
\bvolume{13(3)},
\bfpage{435}--\blpage{455}
(\byear{1997})
\end{barticle}
\endOrigBibText
\bptok{structpyb}%
\endbibitem

\bibitem{NP-book}
\begin{bbook}
\bauthor{\bsnm{Nourdin}, \binits{I.}},
\bauthor{\bsnm{Peccatti}, \binits{G.}}:
\bbtitle{Normal Approximations with Malliavin Calculus From Stein's Method to
 Universality}.
\bpublisher{Cambridge University Press},
\blocation{Cambridge}
(\byear{2012})
\bid{doi={10.1017/CBO9781139084659}, mr={2962301}}
\end{bbook}
%
\OrigBibText
\begin{bbook}
\bauthor{\bsnm{Nourdin}, \binits{I.}},
\bauthor{\bsnm{Peccatti}, \binits{G.}}:
\bbtitle{Normal Approximations with Malliavin Calculus From Stein's Method to
 Universality}.
\bpublisher{Cambridge university press},
\blocation{Cambridge}
(\byear{2012})
\end{bbook}
\endOrigBibText
\bptok{structpyb}%
\endbibitem

\bibitem{NNT}
\begin{barticle}
\bauthor{\bsnm{Nourdin}, \binits{I.}},
\bauthor{\bsnm{Nualart}, \binits{D.}},
\bauthor{\bsnm{Tudor}, \binits{C.}}:
\batitle{Central nd non-central limit theorems for weighted power variations of
 fractional Brownian motion}.
\bjtitle{Ann. Inst. Henri Poincar\'{e}}
\bvolume{46}(\bissue{4}),
\bfpage{1055}--\blpage{1079}
(\byear{2010})
\bid{doi={10.1214/09-AIHP342}, mr={2744886}}
\end{barticle}
%
\OrigBibText
\begin{barticle}
\bauthor{\bsnm{Nourdin}, \binits{I.}},
\bauthor{\bsnm{Nualart}, \binits{D.}},
\bauthor{\bsnm{Tudor}, \binits{C.}}:
\batitle{Central nd non-central limit theorems for weighted power variations of
 fractional brownian motion}.
\bjtitle{Annales de l\' Institut H. Poincar\'e}
\bvolume{46(4)},
\bfpage{1055}--\blpage{1079}
(\byear{2010})
\end{barticle}
\endOrigBibText
\bptok{structpyb}%
\endbibitem

\bibitem{PoTr}
\begin{barticle}
\bauthor{\bsnm{Pospisil}, \binits{J.}},
\bauthor{\bsnm{Tribe}, \binits{R.}}:
\batitle{Parameter estimates and exact variations for stochastic heat equation
 heat equations driven by space-time white noise}.
\bjtitle{Anal. Appl.}
\bvolume{25}(\bissue{3}),
\bfpage{593}--\blpage{611}
(\byear{2007})
\bid{doi={10.1080/07362990701282849}, mr={2321899}}
\end{barticle}
%
\OrigBibText
\begin{barticle}
\bauthor{\bsnm{Pospisil}, \binits{J.}},
\bauthor{\bsnm{Tribe}, \binits{R.}}:
\batitle{Parameter estimates and exact variations for stochastic heat equation
 heat equations driven by space-time white noise}.
\bjtitle{Analysis and Applications 25}
\bvolume{25(3)},
\bfpage{593}--\blpage{611}
(\byear{2007})
\end{barticle}
\endOrigBibText
\bptok{structpyb}%
\endbibitem

\bibitem{RuTu}
\begin{barticle}
\bauthor{\bsnm{Russo}, \binits{F.}},
\bauthor{\bsnm{Tudor}, \binits{C.A.}}:
\batitle{On bifractional Brownian motion}.
\bjtitle{Stoch. Process. Appl.}
\bvolume{5},
\bfpage{830}--\blpage{856}
(\byear{2006})
\bid{doi={10.1016/j.spa.2005.11.013}, mr={2218338}}
\end{barticle}
%
\OrigBibText
\begin{barticle}
\bauthor{\bsnm{Russo}, \binits{F.}},
\bauthor{\bsnm{Tudor}, \binits{C.A.}}:
\batitle{On bifractional brownian motion}.
\bjtitle{Stochastic Processes and their Applications}
\bvolume{no.5},
\bfpage{830}--\blpage{856}
(\byear{2006})
\end{barticle}
\endOrigBibText
\bptok{structpyb}%
\endbibitem

\bibitem{T}
\begin{bbook}
\bauthor{\bsnm{Tudor}, \binits{C.}}:
\bbtitle{Analysis of Variations for Self-similar Processes. A Stochastic
 Calculus Approach. Probability and Its Applications}.
\bpublisher{Springer},
\blocation{New York}
(\byear{2013})
\bid{doi={10.1007/978-3-319-00936-0}, mr={3112799}}
\end{bbook}
%
\OrigBibText
\begin{bbook}
\bauthor{\bsnm{Tudor}, \binits{C.}}:
\bbtitle{Analysis of Variations for Self-similar Processes. A Stochastic
 Calculus Approach. Probability and Its Applications}.
\bpublisher{Springer},
\blocation{New York}
(\byear{2013})
\end{bbook}
\endOrigBibText
\bptok{structpyb}%
\endbibitem

\bibitem{ZZ}
\begin{barticle}
\bauthor{\bsnm{Zili}, \binits{M.}},
\bauthor{\bsnm{Zougar}, \binits{E.}}:
\batitle{Exact variations for stochastic heat equations with piecewise constant
 coefficients and applications to parameter estimation}.
\bjtitle{Teor. \u{I}mov\={\i}r. Mat. Stat.}
\bvolume{1}(\bissue{100}),
\bfpage{75}--\blpage{101}
(\byear{2019})
\end{barticle}
%
\OrigBibText
\begin{botherref}
\oauthor{\bsnm{Zili}, \binits{M.}},
\oauthor{\bsnm{Zougar}, \binits{E.}}:
Exact variations for stochastic heat equations with piecewise constant
 coefficients and applications to parameter estimation.
(2019)
\end{botherref}
\endOrigBibText
\bptok{structpyb}%
\endbibitem

\end{thebibliography}
\end{document}